\documentclass[a4paper, reqno]{amsart}
\usepackage[utf8]{inputenc}
\usepackage{graphicx}
\usepackage[bookmarksnumbered,colorlinks,plainpages]{hyperref}
\usepackage{rotating}
\usepackage{float}

\usepackage{amsthm}
\usepackage[all,2cell]{xy}
\usepackage{amsfonts}
\usepackage{amsmath}
\usepackage{amssymb}
\usepackage{xypic,leqno,amscd,amssymb,pstricks,latexsym,amsbsy,xypic,mathrsfs,verbatim}
\usepackage{multirow}
\usepackage{booktabs}
\usepackage{makecell}
\usepackage{xcolor}

\usepackage{etex}
\usepackage{mathtools}

\usepackage{tikz}
\usetikzlibrary{positioning,arrows}
\usetikzlibrary{decorations.pathreplacing}
\usepackage{tikz-cd}

\UseAllTwocells \SilentMatrices
\usepackage{amscd,amssymb,pstricks,latexsym,amsbsy,xypic,mathrsfs,verbatim}
\usepackage[all]{xy}
\usepackage{lscape}
\usepackage{enumerate}
\usepackage{graphicx}
\usepackage{bm}

\usepackage{amsmath}
\usepackage{xypic,leqno,amscd,amssymb,pstricks,latexsym,amsbsy,xypic,mathrsfs,verbatim}
\usepackage{multirow}
\usepackage{booktabs}
\usepackage{makecell}
\usepackage{mathrsfs}

\makeatletter
\newcommand*\bigcdot{\mathpalette\bigcdot@{.5}}
\newcommand*\bigcdot@[2]{\mathbin{\vcenter{\hbox{\scalebox{#2}{$\m@th#1\bullet$}}}}}
\makeatother

\makeatletter
\@namedef{subjclassname@2010}{
  \textup{2020} Mathematics Subject Classification}

\numberwithin{equation}{section}

\theoremstyle{plain}
\newtheorem{prop}{Proposition}[section]
\newtheorem{thm}[prop]{Theorem}
\newtheorem{cor}[prop]{Corollary}
\newtheorem{lem}[prop]{Lemma}
\newtheorem{defn}[prop]{Definition}

\theoremstyle{definition}

\newtheorem{rem}[prop]{Remark}
\newtheorem{assump}[prop]{Assumption}

\def\Hom{{{\rm Hom}\,}}

\def\add{{{\rm add}\,}}

\def\Aut{{{\rm Aut}\,}}

\def\coh{{\rm coh\,}}
\def\vect{{\rm vect\,}}

\def\add{{\rm add}}

\def\Aut{{ \rm Aut }}

\def\Hom{{\rm Hom}}

\def\Aut{{\rm Aut}}

\begin{document}

\title[{Semi-orthogonal decompositions via t-stabilities}]
{Semi-orthogonal decompositions via t-stabilities}

\author[Mingfa Chen] {Mingfa Chen}

\address{Mingfa Chen;\; Department of Mathematical Sciences, Tsinghua University, Beijing 100084, China.}
\email{\textbf{\large mfchen-tsinghua@outlook.com}}


\date{\today}
\makeatletter \@namedef{subjclassname@2020}{\textup{2020} Mathematics Subject Classification} \makeatother

\begin{abstract}
In this paper we introduce a local-refinement procedure to investigate finite t-stabilities on a triangulated category, and show a direct sufficient condition for a finite t-stability to be finite finest.
We classify all finite finest t-stabilities for certain triangulated categories, including those from the projective plane, weighted projective lines, and finite acyclic quivers.
As applications, we obtain a simple classification of semi-orthogonal decompositions (SOD) for these categories.
Furthermore, we study the connectedness of SODs by a reduction method and give an easily verified criterion for it.
\end{abstract}

\maketitle
\section{Introduction}
Semi-orthogonal decompositions (SOD for short) introduced by Bondal and Orlov in \cite{bono} provide a powerful tool to study the minimal model program
and become a cornerstone in birational geometry and derived categories.
Bergh and Schn$\mathrm{\ddot{u}}$rer \cite{bers} developed a conservative descent technique to study SODs in abstract triangulated categories and algebraic stacks via relative Fourier--Mukai transforms.
Toda \cite{toda} constructed a class of SODs for Pandharipande--Thomas stable pair moduli spaces on Calabi--Yau 3-folds, categorifying the wall-crossing formula for Donaldson--Thomas invariants.
Kuznetsov and Perry \cite{kuzp} showed that categorical joins of moderate Lefschetz varieties inherit a SOD (Lefschetz decomposition), and this construction is compatible with homological projective duality.
It turns out that semi-orthogonal decompositions play a significant role in the study of Bridgeland stability condition, Hochschild homology and cohomology, Hodge theory, mirror symmetry, motives, see, for example, \cite{anel,bei,belm,bers,bodb,cp,kite,lpz,oka}.

Admissible subcategories are a key component of the construction of SODs.
Pirozhkov \cite{pir} proved that admissible subcategories of the projective plane are generated by exceptional sequences and that del Pezzo surfaces admit no such subcategories with the trivial Grothendieck group, known as phantom subcategories.
Krah \cite{krah} constructed a blow-up of the projective plane with both an exceptional sequence and a phantom subcategory, providing a counterexample to the conjectures posted by Kuznetsov \cite{kuzn} and Orlov \cite{orl}.

It seems a lack of detailed description of SODs in the literature for triangulated categories arising from representation theory of algebras, such as the derived category of coherent sheaves on a weighted projective line introduced by Geigle and Lenzing \cite{gl}, which is derived equivalent to Ringel's canonical algebra \cite{ring}.
This motivates our study of SODs in these categories.

The notion of stability data, originating from Geometric Invariant Theory and introduced by Mumford \cite{mum}, was initially used to construct moduli spaces of vector bundles on algebraic curves.
Gorodentsev, Kuleshov, and Rudakov \cite{gkr} later generalized this to
t-stability on triangulated categories, extending Gieseker's stability for torsion-free sheaves and enabling the study of various moduli problems.
They showed that t-stability effectively classifies bounded $t$-structures for the projective line and elliptic curves.
Bridgeland \cite{tbs,tbsp} extended this to stability conditions, while recent work \cite{rw} highlights the role of t-stability in proving the existence of t-exceptional triples and the connectedness and contractibility of Bridgeland's stability space for weighted projective lines.

The aim of this paper is to study finite finest t-stabilities on triangulated categories and apply them to obtain a simple classification of SODs.
We define a partial ordering for the set of finite t-stabilities on a triangulated category, minimal elements with respect to this partial ordering will be called the finite finest t-stabilities.
We classify all finite finest t-stabilities for triangulated categories with a Serre functor under a certain assumption.
This gives a direct classification of SODs for these categories.
Moreover, we study the connectedness of the mutation graph of SODs and give a criterion to determine when it is connected.

The paper is organized as follows.
Section 2 collects the definitions and properties of SODs and admissible subcategories.
In Section 3, we study the relationship between finite t-stabilities and SODs, introduce partial orders on their sets, and describe a procedure for their local refinement.
Section 4 investigates the mutation relation between  partially ordered SODs and admissible filtrations and shows a bijection between them.
In Section 5, we obtain a classification of SODs in terms of exceptional sequences in a triangulated category with a Serre functor.
Section 6 shows reduction methods for studying SODs and their mutation graphs.
Finally, Section 7 presents examples.

\emph{Notations.}
Throughout this paper, we fix an algebraically closed field $\textbf{k}$.
A triangulated category $\mathcal{D}$ is always assumed to be $\textbf{k}$-linear.
We use $\Phi_{n}$ to denote the linearly ordered set $\{1, 2, \cdots,n\}$ in the sense that $1<2<\cdots <n$.
Given a set $\mathcal{S}$ of objects in $\mathcal{D}$, we write $\langle \mathcal{S} \rangle$ for the full triangulated subcategory of $\mathcal{D}$ generated by the objects in $\mathcal{S}$ closed \emph{under extensions,
direct summands and shifts}.

Let $\mathcal{B}\subset\mathcal{D}$ be a full triangulated subcategory. The right orthogonal $\mathcal{B}^{\bot}$ to $\mathcal{B}$ is
$$\mathcal{B}^{\bot}:=\{X\in\mathcal{D}\mid \Hom(B,X)=0\;\,\text{for all}\;\,B\in\mathcal{B}\;\,\}.$$
The left orthogonal ${}^{\bot}\mathcal{B}$ is defined similarly.

\section{Semi-orthogonal decompositions and admissible subcategories}
This section recalls properties of SODs and admissible subcategories, and makes explicit a bijection between admissible SODs and finite strongly admissible filtrations previously implicit in \cite{bonk}.

\begin{defn}[\cite{bono}]\label{defn of SOD}
A \emph{semi-orthogonal decomposition} (\emph{SOD} for short) of a triangulated category $\mathcal{D}$ is a sequence of full triangulated
subcategories $\Pi _{1}, \Pi _{2}, \cdots, \Pi _{n}$ such that
\begin{itemize}
   \item [(1)] $\mathrm{Hom}(\Pi _{j},\Pi_{i})=0$ for any $1\leq i<j\leq n$;
  \item [(2)] $\mathcal{D}=\langle\Pi_{i}\mid i\in\Phi_n\rangle$.
\end{itemize}
\end{defn}

By definition, $\mathcal{D}$ admits two trivial SODs, $(\mathcal{D},0)$ and $(0,\mathcal{D})$; all others are called non-trivial and are considered in this paper. We write a SOD as $(\Pi_i;i\in\Phi_n)$, where for each $i$,
$$\Pi_i = {}^\bot\langle \Pi_1,\dots,\Pi_{i-1} \rangle \cap \langle \Pi_{i+1},\dots,\Pi_n \rangle^\bot.$$

\begin{defn}[\cite{bonk}]
A full triangulated subcategory $\mathcal{A}\subset\mathcal{D}$ is called \emph{right admissible} if, for the inclusion
functor $i:\mathcal{A}\hookrightarrow\mathcal{D}$, there is a right adjoint $i^{!}:\mathcal{D}\rightarrow\mathcal{A}$, and \emph{left admissible} if there is a left adjoint
$i^{*}:\mathcal{D}\rightarrow\mathcal{A}$. It is called \emph{admissible} if it is both left and right admissible.
\end{defn}

\begin{lem} [\cite{bond,bonk}]\label{right addmiss and sod}
Let $\mathcal{D}$ be a triangulated category.
\begin{itemize}
  \item [(1)] If $(\mathcal{A},\mathcal{B})$ is a SOD of $\mathcal{D}$, then $\mathcal{A}$ is left admissible and $\mathcal{B}$ is right admissible.
  \item [(2)] If $\mathcal{A}\subset\mathcal{D}$ is left admissible, then $(\mathcal{A},{}^{\bot}\mathcal{A})$ is a
SOD and $({}^{\bot}\mathcal{A})^{\bot}=\mathcal{A}$, and if $\mathcal{B}\subset\mathcal{D}$ is right admissible, then $(\mathcal{B}^{\bot},\mathcal{B})$ is a SOD and ${}^{\bot}(\mathcal{B}^{\bot})=\mathcal{B}$.
\end{itemize}
\end{lem}

Recall from \cite{bonk} that a finite sequence of full triangulated subcategories
\begin{equation}\label{admissble subcats seq}
\mathscr{T}:\;
0=\mathcal{T}_0\subsetneq\mathcal{T}_1\subsetneq\cdots\subsetneq\mathcal{T}_{n-1}
\subsetneq\mathcal{T}_{n}=\mathcal{D}
\end{equation} is called (\emph{right, left}) \emph{admissible} if each $\mathcal{T}_i$ is (right, left) admissible as a full subcategory of  $\mathcal{T}_{i+1}$ for $1\leq i\leq n-1$.
Furthermore, the finite right admissible filtration is called \emph{strongly admissible} (\emph{s-admissible} for short) if each $\mathcal{T}^{\bot}_{i}\cap\mathcal{T}_{i+1}$ is admissible in $\mathcal{D}$.
Similarly, a SOD $(\Pi_i;i\in\Phi_n)$ is called
\emph{admissible} if each $\Pi_{i}$ is admissible, cf. \cite{bono}.
We remark that the notion of an admissible SOD coincides with that of a finite directed preordered SOD under the reformulated definition given in \cite{sst}.

Given a SOD $(\Pi_{i};i\in\Phi_n)$ of $\mathcal{D}$, we know that $(\Pi_{\leq i-1},\Pi_{\geq i})$ forms a SOD for each $1\leq i\leq n$, with $\Pi_{\leq i-1}:=\langle \Pi_{j}\mid j\leq i-1\rangle$ and $\Pi_{\geq i}:=\langle \Pi_{j}\mid j\geq i\rangle$. By Lemma \ref{right addmiss and sod}, each $\Pi_{\geq i}$ is right admissible, and there exists a right admissible filtration:
\begin{equation}\label{defn of induced chain of right admissible subcategories}
0\subsetneq\Pi_{\geq n }\subsetneq\cdots\subsetneq\Pi_{\geq 2 }\subsetneq
\Pi_{\geq 1}=\mathcal{D}.
\end{equation}
This defines a map
$$\xi: \{{\text{SODs of }} \mathcal{D}\}\longrightarrow \{\text{finite right admissible filtrations in } \mathcal{D}\}.$$

The following result, implicitly stated in \cite{bonk}, will be used.

\begin{prop} \label{finit t-stab and finite right admiss}
Keep notations as above.
Then $\xi$ is bijective and restricts to a bijection:
$$\begin{array}{lcr}
\xi:\left\{
\mbox{admissible SODs of $\mathcal{D}$}
\right\}&
\stackrel{1:1}{\longleftrightarrow}&
\left\{ \mbox{finite s-admissible filtrations in $\mathcal{D}$} \right\}
\end{array}.$$
\end{prop}

\begin{proof}
We construct the inverse map of $\xi$ as follows. Let
$\mathscr{T}:\;
0=\mathcal{T}_0\subsetneq\mathcal{T}_1\subsetneq\cdots\subsetneq\mathcal{T}_{n-1}
\subsetneq\mathcal{T}_{n}=\mathcal{D}$
be a finite right  admissible filtration in $\mathcal{D}$.
Denote by \begin{equation*} \label{map between finite admiss subcat seq and sod}
\Pi_{n}=\mathcal{T}_1 \quad\text{and}\quad \Pi_{i}=\mathcal{T}^{\bot}_{n-i}\cap\mathcal{T}_{n-i+1}\quad
\text{for}\quad1\leq i\leq n-1.
\end{equation*}
By \cite[Lem. 3.1]{bond}, the right admissibility of
$\mathcal{T}_{i}$  in $\mathcal{T}_{i+1}$ implies that
$\mathcal{T}_{i+1}=\langle \mathcal{T}^{\bot}_{i}\cap\mathcal{T}_{i+1},\mathcal{T}_{i} \rangle$
for $1\leq i\leq n-1$. Thus, $\mathcal{T}_{i}=\langle \Pi_{n-i+1},\Pi_{n-i+2}, \cdots, \Pi_{n}\rangle$
for $1\leq i\leq n$. From the construction of $\Pi_{i}$ it follows that
$\Hom(\Pi_{j},\Pi_{i})=0$ for any $1\leq i< j\leq n$ and
$\mathcal{D}=\mathcal{T}_{n}=\langle \Pi_{1}, \Pi_{2}, \cdots, \Pi_{n}\rangle$.
Consequently, $(\Pi_i;i\in\Phi_n)$ is a SOD. One can easily check that this gives the inverse map of $\xi$.

Now we prove the second statement. Assume that $(\Pi_i;i\in\Phi_n)$ is an admissible SOD.
Then the induced filtration \eqref{defn of induced chain of right admissible subcategories} is finite s-admissible since $\Pi^{\bot}_{\geq i+1}\cap\Pi_{\geq i}=\Pi_i$ is admissible in $\mathcal{D}$.
Conversely, assume the filtration \eqref{admissble subcats seq} is finite s-admissible. Then the induced  SOD $(\Pi_i;i\in\Phi_n)$ is admissible since $\Pi_{i}=\mathcal{T}^{\bot}_{n-i}\cap\mathcal{T}_{n-i+1}$ is admissible in $\mathcal{D}$.
\end{proof}

Dually, for a given SOD $(\Pi_{i};i\in\Phi_n)$,
each $\Pi_{\leq i}$ is left admissible, and
there exists a left admissible filtration:
\label{perpend dual}
\begin{equation}\label{defn of induced chain of left admissible subcategories}
0\subsetneq\Pi_{\leq 1}\subsetneq\cdots\subsetneq\Pi_{\leq 2}\subsetneq
\Pi_{\leq n}=\mathcal{D}.
\end{equation}
This defines a bijection between SODs and finite left admissible filtrations
in  $\mathcal{D}$. Moreover, these two filtrations
\eqref{defn of induced chain of right admissible subcategories} and
\eqref{defn of induced chain of left admissible subcategories} are
related to each other by taking right and left perpendicular, see for example \cite[Prop. 2.6]{bodb}.

\section{Partial orders on finite t-Stabilities and SODs}

In this section we recall t-stability for  triangulated categories and show a one-to-one correspondence from finite t-stabilities to SODs.
Furthermore, we introduce partial orders on the sets of all finite t-stabilities and SODs, and provide an explicit refinement procedure for them.

\subsection{Finite t-stabilities}

Introduced by Gorodentsev--Kuleshov--Rudakov, a key property of t-stability is that every object of $\mathcal{D}$ admits a Postnikov tower with ordered semistable factors.

\begin{defn} [{\rm \cite[Def. 3.1]{gkr}}] \label{defn of t-stability}
Suppose that $\mathcal{D}$ is a triangulated category, $\Phi$ is a linearly ordered set, and a strictly full extension-closed non-trivial subcategory
$\Pi_\varphi\subset\mathcal{D}$ is given for every $\varphi\in\Phi$. A pair
$(\Phi,\{\Pi_\varphi\}_{\varphi\in\Phi})$ is called a \emph{t-stability}  on $\mathcal{D}$ if
\begin{itemize}
  \item [(1)] the grading shift functor $X\mapsto X[1]$ acts on $\Phi$ as a non-decreasing automorphism, that is, there is a bijection
      $\tau_{\Phi}\in\Aut(\Phi)$ such that $\Pi_{\tau_{\Phi}(\varphi)}=\Pi_{\varphi}[1]$ and $\tau_{\Phi}(\varphi)\geq\varphi$ for
      all $\varphi$;
  \item [(2)] ${\rm{Hom}}^{\leq0}(\Pi _{\varphi'},\Pi_{\varphi''})=0$ for all $\varphi'>\varphi''$ in $\Phi$;
  \item [(3)] every non-zero object $X\in\mathcal{D}$  admits a finite sequence of triangles
\begin{equation}\label{HN-filt2.1}
\xymatrix @C=6.55em@R=10ex@M=4pt@!0{ 0=X_{n}\ar@{->}[r]^-{p_{n}}\ar@{--}[rd] &X_{n-1}\ar@{->}[d]^{q_{n-1}}\ar@{->}[r]^-{p_{n-1}}\ar@{--}[rd] &X_{n-2}
\ar@{->}[d]^-{q_{n-2}}\ar@{->}[r]^-{p_{n-2}}
&\cdots\ar@{->}[r]^-{p_{2}}&X_{1}\ar@{->}[d]^-{q_{1}}
\ar@{->}[r]^-{p_{1}}\ar@{--}[rd]&X_0=X\ar@{->}[d]^{q_{0}}&\\
& A_n&A_{n-1}&&A_{2}&A_{1}\\ }
\end{equation}
with non-zero factors $A_{i}=\mathrm{cone}\; (p_{i})\in\Pi_{\varphi_i}$ and
strictly decreasing $\varphi _{i+1}>\varphi_{i}.$
\end{itemize}
\end{defn}

It has been shown in \cite[Thm. 4.1]{gkr} that for each non-zero object $X$,
the decomposition $\eqref{HN-filt2.1}$ is unique up to isomorphism, which
is known as the \emph{Harder--Narasimhan filtration} (\emph{HN-filtration} for short)of $X$.
The factors $A_i$ are called the \emph{semistable factors} of $X$, and each $\Pi_{\varphi}$ is called a \emph{semistable subcategory}. Furthermore, every semistable subcategory $\Pi_{\varphi}$ is closed under direct summands. Finally, we have $q_i \neq 0$ for $0 \leq i \leq n-1$ and $p_i \cdots p_1 \neq 0$ for $1 \leq i \leq n-1$.

We remark that under condition $(1)$ in Definition \ref{defn of t-stability}, the statement $(2)$ above is equivalent to the following (cf. \cite{rw}):
\begin{itemize}
\item [(2)$'$] ${\rm{Hom}}^{}(\Pi _{\varphi'},\Pi_{\varphi''})=0$ for all $\varphi'>\varphi''$ in $\Phi$;
\end{itemize}
moreover, if $\Phi$ is a finite set, then  $\tau_{\Phi}\in\Aut(\Phi)$ is a order-preserving
bijection which implies $\tau_{\Phi}=\text{id}$.

Let us now call a t-stability $(\Phi,\{\Pi_{\varphi}\}_{\varphi\in\Phi})$ \emph{finite} if $\Phi$ is a finite set.
Furthermore, we refer to such a t-stability restricted to a triangulated subcategory $\mathcal{B} \subset \mathcal{D}$ as a \emph{local finite t-stability} on $\mathcal{B}$.

\begin{rem}
We emphasize that for a t-stability $(\Phi, \{\Pi_\varphi\}_{\varphi \in \Phi})$ with $\Phi$ an infinite set, the semistable subcategories $\Pi_\varphi$ are not necessarily triangulated.

Moreover, we see that a t-stability coincides, in certain cases, with a preordered SOD in the sense of \cite[Def. 3.1]{sst}.
In fact, let $(\Phi, \{\Pi_\varphi\}_{\varphi\in\Phi})$ be a t-stability such that each $\Pi_\varphi$ is admissible. Then $(\mathcal{D}, \Phi)$ is a preordered SOD.
Conversely, let $(\mathcal{D}, \mathcal{P})$ be a preordered SOD with $\mathcal{P}$ linearly ordered. Then $(\mathcal{P}, \{\mathcal{D}_x\}_{x \in \mathcal{P}})$ is a t-stability with every $\mathcal{D}_x$ admissible.
\end{rem}

\subsection{Another description of HN-filtrations and a bijection}

In the following we show a construction of the HN-filtration, a key ingredient of which is the use of triangulated semistable subcategories that can indeed be indexed by an infinite set, compare with \cite[Lem. 3.1]{bond}.

\begin{prop} \label{finite id t-stab}
Assume the pair $(\Phi,\{\Pi_{\varphi}\}_{\varphi\in\Phi})$ satisfies the
conditions $(1)$ and $(2)$ in Definition \ref{defn of t-stability} with $\tau_{\Phi}=\mathrm{id}$.
Then
$(3)$ is equivalent to the following:
\begin{itemize}
\item [(3)$'$] $\mathcal{D}=\langle\Pi_{\varphi}\mid\varphi\in\Phi\rangle.$
\end{itemize}
\end{prop}

\begin{proof}
It suffices to show that (3)$'$ implies (3) since the converse is straightforward.
For this we assume  $0\neq X\in\mathcal{D}=\langle\Pi_{\varphi}\mid\varphi\in\Phi\rangle$.
Then $X$ admits a finite sequence of triangles
\begin{equation}\label{seq of reorder triangles}
\xymatrix @C=6.55em@R=10ex@M=4pt@!0{ 0=X_{n}\ar@{->}[r]^-{p_{n}}\ar@{--}[rd] &X_{n-1}\ar@{->}[d]^{q_{n-1}}\ar@{->}[r]^-{p_{n-1}}\ar@{--}[rd] &X_{n-2}
\ar@{->}[d]^-{q_{n-2}}\ar@{->}[r]^-{p_{n-2}}
&\cdots\ar@{->}[r]^-{p_{2}}&X_{1}\ar@{->}[d]^-{q_{1}}
\ar@{->}[r]^-{p_{1}}\ar@{--}[rd]&X_0=X\ar@{->}[d]^{q_{0}}&\\
&A_{n}&A_{n-1}&&A_{2}&A_{1}\\ }
\end{equation}
with non-zero factors $A_{i}=\mathrm{cone}\; (p_{i})\in\Pi_{\varphi_i}$.

Consider the following cases:

{\bf Case 1}: Assume $\varphi _{i+1}=\varphi_{i}$ for some $i$.
By \eqref{seq of reorder triangles} we have a triangle
$\xi_{i+1}: X_{i+1}\xrightarrow{p_{i+1}} X_{i}\xrightarrow {q_{i}} A_{i+1}\rightarrow X_{i+1}[1]$. Taking the octahedral axiom along  $p_{i}$,
we obtain the following commutative diagram
$$\xymatrix{
   & & A_{i}[-1] \ar[d]_{}\ar@{=}[r]^{}   & A_{i}[-1]\ar[d]_{} &  & & \\
     & X_{i+1}\ar@{=}[d]_{}  \ar[r]^{p_{i+1}} &  X_{i} \ar[d]^-{p_{i}} \ar[r]^{} & A_{i+1} \ar[d]^-{}\ar[r]^{} & X_{i+1}[1]\ar@{=}[d]_{}  & &   \\
   &X_{i+1}  \ar[r]^-{p_{i}p_{i+1}} & X_{i-1}  \ar[r]^-{} \ar[d]^-{}& B_{i+1}\ar[d]^{} \ar[r]^{} & X_{i+1}[1].  & &  \\
   &  & A_{i} \ar@{=}[r]^{} &A_{i}  &   & &  \\
 }$$
Then we have the following triangle:
$$\xi'_{i+1,i-1}: X_{i+1}\rightarrow X_{i-1}\rightarrow B_{i+1}\rightarrow X_{i+1}[1],$$
with  $B_{i+1}\in\langle A_{i+1}, A_{i}\rangle\subseteq\Pi_{\varphi_{i}}=\Pi_{\varphi_{i+1}}$.
Consequently, the sequence of triangles for $X$ has the form:
\begin{equation*}
\xymatrix @C=6.05em@R=10ex@M=4pt@!0{
0=X_{n}\ar@{->}[r]^-{p_{n}}\ar@{--}[rd] &
X_{n-1}\ar@{->}[d]^{q_{n-1}}\ar@{->}[r]^-{p_{n-1}}&
\cdots\ar@{->}[r]&
X_{i+1}\ar@{->}[r]^-{p_{i}p_{i+1}}\ar@{--}[rd]&
X_{i-1}\ar@{->}[d]^-{}\ar@{->}[r]^-{}&
\cdots\ar@{->}[r]^-{p_{2}}&X_{1}\ar@{->}[d]^-{q_{1}}
\ar@{->}[r]^-{p_{1}}\ar@{--}[rd]&X_0=X\ar@{->}[d]^{q_{0}}&\\
&A_{n}&&&B_{i+1}&&A_{2}&A_{1}\\ }
\end{equation*}
Therefore, we can assume $\varphi _{i+1}\neq\varphi_{i}$ for any $i$.

{\bf Case 2}:
Assume $\varphi _{i+1}<\varphi_{i}$ for some $i$.
Note that $\tau_{\Phi}=\text{id}$ and hence $\Pi_{\tau_{\Phi}(\varphi)}=\Pi_{\varphi}=\Pi_{\varphi}[1]$ for any $\varphi\in\Phi$.
It follows that $\Hom^{\bullet}(A_{i},A_{i+1})\subset
\Hom(\Pi_{\varphi_{i}},\Pi_{\varphi_{i+1}})=0$,
and an argument analogous to Case 1 yields  the triangle $X_{i+1}\xrightarrow{p_{i}p_{i+1}} X_{i-1}\xrightarrow {} A_{i+1}\oplus A_{i}\rightarrow X_{i+1}[1]$.
Applying the octahedral axiom along embedding $A_{i}\hookrightarrow A_{i+1}\oplus A_{i}$,
we obtain the following commutative diagram
$$\xymatrix{
   & & A_{i+1}[-1] \ar[d]_{}\ar@{=}[r]^{}   & A_{i+1}[-1]\ar[d]^{} &  & & \\
     & X_{i+1}\ar@{=}[d]_{}  \ar[r]^{p'_{i+1}} & Y_{i} \ar[d]^-{p'_{i}} \ar[r]^{q'_{i}} & A_{i} \ar[d]^-{}\ar[r]^{} & X_{i+1}[1]\ar@{=}[d]_{}  & &   \\
   &X_{i+1}  \ar[r]^-{} & X_{i-1}  \ar[r]^-{} \ar[d]^-{q'_{i-1}}& A_{i+1}\oplus A_{i}\ar[d]^{} \ar[r]^{} & X_{i+1}[1].  & &  \\
   &  & A_{i+1} \ar@{=}[r]^{} &A_{i+1}  &   & &  \\
 }$$
This gives the following two triangles:
$$\xi'_{i+1}: X_{i+1}\xrightarrow{p'_{i+1}} Y_{i}\xrightarrow{q'_{i}} A_{i}\rightarrow X_{i+1}[1] \quad\text{and}\quad \xi'_{i}: Y_{i}\xrightarrow{p'_{i}} X_{i-1}\xrightarrow{q'_{i-1}} A_{i+1}\rightarrow Y_{i+1}[1], $$
which fit together as follows:
\begin{equation*}
\xymatrix@C=8.75em@R=10ex@M=4pt@!0 { X_{i+1}\ar@{->}[r]^-{p'_{i+1}}\ar@{--}[rd] &Y_{i}\ar@{->}[d]^-{q'_{i}}\ar@{->}[r]^-{p'_{i}}\ar@{--}[rd]
&X_{i-1}.\ar@{->}[d]^-{q'_{i-1}}\\
&A_{i}&A_{i+1} }
\end{equation*}

Keeping the procedure going on step by step, we finally obtain that  $X$ admits
a filtration with cones, namely $(B_{m}, B_{m-1}, \cdots, B_{1})$, having strictly decreasing order in $\Phi$.
Hence, the sequence of triangles for $X$ in fact fits into a HN-filtration.
This completes the proof.
\end{proof}

\begin{defn} \label{equiv of finite t-stab}
Two t-stabilities $(\Phi,\{\Pi_{\varphi}\}_{\varphi\in\Phi}),(\Psi,\{P_{\psi}\}_{\psi\in\Psi})$ on triangulated category $\mathcal{D}$ are called \emph{equivalent}  if there exists an order-preserved bijective map $r:\Phi\rightarrow\Psi$ such that  $P_{r(\varphi)}=\Pi_\varphi$ for any $\varphi\in\Phi$.
\end{defn}

Thus, up to equivalence, any finite t-stability can be presented by $(\Phi_n,\{\Pi_{i}\}_{i\in\Phi_n})$ for some $n\in\mathbb{N}$.

For any interval $I\subseteq \Phi_n$,
we define $\Pi_{I}:=\langle\Pi_{i}\mid i\in I\rangle$ to be the triangulated subcategory of $\mathcal{D}$ generated by $\Pi_{i}$ for all $i\in I$.
Then, given a finite t-stability $(\Phi_n,\{\Pi_{i}\}_{i\in\Phi_n})$ on $\mathcal{D}$, one can note that
non-zero objects in $\Pi_{I}$ are exactly those objects  $X\in\mathcal{D}$ which satisfy $\bm \phi^{\pm}(X)\in I$.

The following result shows a precise correspondence between finite t-stabilities and SODs.

\begin{prop}\label{t-stab and sod}
Let $\mathcal{D}$ be a triangulated category. There is a bijection
\begin{eqnarray*}
\eta: \{\mbox{equvalence classes of finite t-stabilities on $\mathcal{D}$}\}&\longrightarrow&
\{\mbox{SODs of $\mathcal{D}$}\} \\
(\Phi_n,\{\Pi_i\}_{i\in\Phi_n}) &\longmapsto& (\Pi_i;i\in\Phi_n).
\end{eqnarray*}
\end{prop}

\begin{proof}
Assume that $(\Phi_n,\{\Pi_i\}_{i\in\Phi_n})$ is a finite t-stability on $\mathcal{D}$.
Since $\tau_{\Phi}=\text{id}$ and  $\Pi_i=\Pi_{\tau_{\Phi}(i)}=\Pi_i[1]$,
each $\Pi_i$ is a full triangulated subcategory.
A direct verification shows that $(\Pi_i;i\in\Phi_n)$ is a SOD of $\mathcal{D}$.
This shows that $\eta$ is well-defined.

Conversely, assume that $(\Pi_i;i\in\Phi_n)$ is a SOD of $\mathcal{D}$.
Set $\tau_{\Phi}=\text{id}\in\Aut(\Phi_n)$.
By Proposition \ref{finite id t-stab}, $(\Phi_n,\{\Pi_i\}_{i\in\Phi_n})$ is a finite t-stability on $\mathcal{D}$.
This defines the inverse of $\eta$,
and it follows that $\eta$ is a bijection.
\end{proof}

\subsection{Partial orders and local refinement construction}
Following \cite{gkr}, we introduce partial orders on the sets of all finite t-stabilities and SODs on $\mathcal{D}$ as follows.

\begin{defn}\label{finer or coarser}
Let $(\Phi_n,\{\Pi_{i}\}_{i\in\Phi_n}),(\Psi_m,\{P_{\psi}\}_{\psi\in\Psi_m})$ be
finite t-stabilities on a triangulated category $\mathcal{D}$. We say that a finite t-stability
$(\Phi_n,\{\Pi_{i}\}_{i\in\Phi_n})$
is \emph{finer} than $(\Psi_m,\{P_{\psi}\}_{\psi\in\Psi_m})$, or  $(\Psi_m,\{P_{\psi}\}_{\psi\in\Psi_m})$ is \emph{coarser} $(\Phi_n,\{\Pi_{i}\}_{i\in\Phi_n})$ and write  $(\Phi_n,\{\Pi_{i}\}_{i\in\Phi_n})\preceq
(\Psi_m,\{P_{\psi}\}_{\psi\in\Psi_m})$, if there exists a surjective map $r:\Phi_n\rightarrow\Psi_m$ such that
\begin{itemize}
 \item [(1)] $r\tau_{\Phi}=\tau_{\Psi}r$;
 \item [(2)] $i'>i''$ implies $ r(i')\geq r(i'')$;
 \item [(3)] for any $\psi\in\Psi_m$, $P_{\psi}=\langle \Pi_{i}\mid i\in r^{-1}(\psi) \rangle$.
\end{itemize}

We say that the SOD $(\Pi_i;i\in\Phi_n)$ is \emph{finer} than $(P_{\psi};\psi\in\Psi_m)$ if the corresponding finite t-stability
$(\Phi_n,\{\Pi_{i}\}_{i\in\Phi_n})$
is finer than $(\Psi_m,\{P_{\psi}\}_{\psi\in\Psi_m})$.
\end{defn}
Minimal elements with respect to these partial orderings will be called the \emph{finite finest} t-stabilities, and \emph{finest} SODs, respectively.

In the rest of this section, results for finite t-stabilities extend directly to SODs.
For a finite t-stability, we introduce a local-refinement method as follows, compare with the main result of \cite{sst}.

\begin{prop}\label{local finite refinement construction}
Let $(\Phi_n,\{\Pi_{i}\}_{i\in\Phi_n})$ be a finite t-stability on a triangulated
category $\mathcal{D}$. For any $i\in\Phi_n$, assume $(I_{i},\{P_{{\psi}}\}_{\psi \in I_{i}})$
is a local finite t-stability  on $\Pi_{i}$. Let $\Psi= \bigcup_{i\in\Phi_n } I_{i}$,
which is a linearly ordered set containing each $I_{i}$ as a linearly ordered
subset, and $\psi_{i_2}>\psi_{i_1}$ whenever $\psi_{i_1}\in I_{i_1},\psi_{i_2}\in I_{i_2}$
with $i_2>i_1$.
Then
$(\Psi,\{P_{\psi}\}_{\psi \in \Psi})$ with $\tau_{\Psi}=\mathrm{id}$ is a finite t-stability on $\mathcal{D}$, which is finer than $(\Phi_n,\{\Pi_{i}\}_{i\in\Phi_n})$.
\end{prop}

\begin{proof}
It is obvious that
$\tau_{\Psi}$ satisfies Definition
\ref{defn of t-stability} (1).
By definition, for any $\psi'>\psi'' \in \Psi$, we have $\Hom(P_{\psi'}, P_{\psi''})=0$.
It follows from the construction that $\Pi_{i}=\langle P_{\psi}\mid\psi\in I_{i}\rangle$ for each $i\in\Phi_n$ and thus
\begin{equation*}
\begin{split}
    \mathcal{D}&=\langle \Pi_{i}\mid i\in\Phi_n\rangle
     =\langle P_{\psi}\mid\psi\in I_{i},i\in\Phi_n\rangle
     = \langle P_{\psi}\mid\psi\in\Psi\rangle.
\end{split}
\end{equation*}
Hence, by Proposition \ref{finite id t-stab}, $(\Psi,\{P_{\psi}\}_{\psi\in\Psi})$ is a finite t-stability on $\mathcal{D}$.

Note that $\Psi= \bigcup_{i\in\Phi_n} I_{i}$. This induces a well-defined surjective map $r: \Psi\to \Phi_n$ by setting $r(\psi)=i$ for all $ \psi \in I_{i}$.
Since $\tau_{\Psi}=\text{id}$, we have $r\tau_{\Psi}(\psi)=r(\psi)=\tau_{\Phi}r(\psi)$.
In order to show that $(\Psi,\{P_{\psi}\}_{\psi \in \Psi})$ is finer than $(\Phi_n,\{\Pi_{i}\}_{i\in\Phi_n})$,
it remains to show that the statements (2) and (3) in
Definition \ref{finer or coarser} hold.
In fact,
if $\psi_{i_2}>\psi_{i_1} \in\Psi$, we claim $r(\psi_{i_2})\geq r(\psi_{i_1})$.
Otherwise, write $r(\psi_{i_1})={i_1}$ and $r(\psi_{i_2})={i_2}$ , that is, $\psi_{i_1} \in I_{i_1}$ $\psi_{i_2} \in I_{i_2}$.
It follows that $\psi_{i_2}<\psi_{i_1}$ by definition of ordering in $\Psi$, a contradiction. This proves the claim.
On the other hand, for any $i\in\Phi_n$,
$\Pi_{i}=\langle P_{{\psi}}\mid \psi \in I_{i}\rangle=\langle P_{{\psi}}\mid \psi \in r^{-1} (i)\rangle$.
This concludes the proof.
\end{proof}

The finite t-stability $(\Psi,\{P_{\psi}\}_{\psi \in \Psi})$ obtained as above will be called a \emph{local refinement} of $(\Phi_n,\{\Pi_{i}\}_{i\in\Phi_n})$.

\subsection{Finite finest t-stabilities}

The following proposition gives a direct sufficient condition for a finite t-stability to be finite finest on a triangulated category.

\begin{prop}\label{necess cond for finest SOD}
Let $\mathcal{D}$ be any triangulated category, and $(\Phi_n,\{\Pi_{i}\}_{i\in\Phi_n})$ a finite t-stability on $\mathcal{D}$.
Suppose that the following condition holds for every $i\in\Phi_n$:
$$\Hom(\langle X\rangle,\langle Y\rangle)\neq 0\neq \Hom(\langle Y\rangle,\langle X\rangle),\quad \forall\;\, 0\neq X,Y\in\Pi_{i}.$$
Then $(\Phi_n,\{\Pi_{i}\}_{i\in\Phi_n})$ is finite finest.
\end{prop}

\begin{proof}
Assume $(\Phi_n,\{\Pi_{i}\}_{i\in\Phi_n})$  is not finite finest. Then
there exists a finite t-stability $(\Psi_m,\{P_{\psi}\}_{\psi\in\Psi_m})$
which is finer than $(\Phi_n,\{\Pi_{i}\}_{i\in\Phi_n})$. Hence, there is a
surjective map $r: \Psi_m \rightarrow \Phi_n$, which is not a bijection.

So there exists $i\in\Phi_n$ and $\psi_2> \psi_1\in\Psi_m$ such that
$r(\psi_1) = r(\psi_2)=i$. Then $P_{\psi_1}, P_{\psi_2} \subseteq \Pi_{i}$.
Moreover, there exist non-zero objects $X,Y\in\Pi_{i}$ such that $X\in P_{\psi_1}, Y\in P_{\psi_2}$.
But $\Hom(\langle Y\rangle, \langle X\rangle)\subset\Hom(P_{\psi_2}, P_{\psi_1})=0$, a contradiction.
\end{proof}

The following result shows a simple classification of SODs using the finite finest t-stability approach.

\begin{prop}\label{any sod construction}
Assume that each finite t-stability on $\mathcal{D}$ is coarser than a finite finest one.
Then for any finite t-stability $(\Psi_m,\{P_\psi\}_{\psi\in\Psi_m})$, there exists a finite
t-stability $(\Phi_n,\{\Pi_{i}\}_{i\in\Phi_n})$ and a decomposition
$\Phi_n=I_{1}\cup I_{2}\cup\cdots\cup I_{m}$ with $\varphi>\varphi'$ for any
$\varphi\in I_i, \varphi'\in I_{i}$ and $i>i'$, such that
$$P_{i}=\langle\Pi_{j}\mid j\in I_{i}\rangle.$$
Moreover, $(\Phi_n,\{\Pi_{i}\}_{i\in\Phi_n})$ can be taken from the set of finite finest t-stabilities.
\end{prop}

\begin{proof}
By assumption,  let $(\Phi_n,\{\Pi_{i}\}_{i\in\Phi_n})$ be a finite finest t-stability which is finer than $(\Psi_m,\{P_{\psi}\}_{\psi\in\Psi_m})$.
Then there is a surjective map $r:\Phi_n\rightarrow\Psi_m$ such that
$$P_i=\langle \Pi_{j}\mid r(j)=i\rangle.$$ Set $I_{i}=\{j\in\Phi_n\mid r(j)=i\}$.
Then the result follows.
\end{proof}

\section{Mutation of SODs and admissible filtrations}  \label{Sect 4}

In this section we study the relationship between mutations of partially ordered SODs and admissible filtrations, and show a one-to-one correspondence between them.

\subsection{Mutations of SODs and  filtrations}

Let $\mathcal{D}$ be a triangulated category.
Recall that a SOD (t-stability, finite filtration) of $\mathcal{D}$ is called admissible if each subcategory appeared therein is admissible.

The notion of SOD mutations, implicit in works such as \cite{bonk} and \cite{kuzn}, has been of great use.

\begin{defn}\label{mutation of admiss sod}
Assume that $(\Pi_i;i\in\Phi_n)$ is a SOD of $\mathcal{D}$.
Fix an integer $1\leq i\leq n-1$ and let
$$\displaystyle \Pi'_{j}= \left\{\begin{array}{l} \Pi^{\bot}_{i}\cap\langle \Pi_{i}, \Pi_{i+1}\rangle,\quad j=i,\\
 \Pi_{i},\quad\quad\quad\quad\quad\quad\;\;\,  j=i+1,\\
\Pi_{j},\quad\quad\quad\quad\quad\quad\;\;\, j\in \Phi_n\setminus\{i,i+1\}. \end{array}\right.$$
Then the SOD $\rho_i(\Pi_i;i\in\Phi_n)=(\Pi'_i;i\in\Phi_n)$ constructed above is called the \emph{right mutation} of $(\Pi_i;i\in\Phi_n)$ at $\Pi_{i}$ for $1\leq i\leq n-1$.
Similarly, one can define \emph{left mutation} $\check{\rho}_i$.
\end{defn}

For an admissible SOD $(\Pi_i;i\in\Phi_n)$, we have
$$\rho_i\check{\rho}_i(\Pi_i;i\in\Phi_n)=
(\Pi_i;i\in\Phi_n)=
\check{\rho}_i\rho_i(\Pi_i;i\in\Phi_n).$$

We call a SOD $\infty$-admissible if all its iterated mutations are admissible.

\begin{defn}[\cite{bonk}]  \label{mutation of finit admiss subcat seq}
Assume that
$\mathscr{T}:\; 0=\mathcal{T}_0\subsetneq\mathcal{T}_1\subsetneq\cdots\subsetneq\mathcal{T}_{n-1}
\subsetneq\mathcal{T}_{n}=\mathcal{D}
$
is a finite  filtration in $\mathcal{D}$.
Fix an integer $1\leq i\leq n-1$ and let
$$\displaystyle \mathcal{T}'_{j}= \left\{\begin{array}{l} \langle  \mathcal{T}^{\bot}_{i}\cap\mathcal{T}_{i+1},
\mathcal{T}_{i-1}\rangle,\quad j=i,\\
\mathcal{T}_{j},\quad\quad\quad\quad\quad\quad\quad\;\;  j\in \Phi_n\setminus\{i\}.\\
 \end{array}\right.$$
Then
$
\mathscr{T}'=\sigma_i\mathscr{T}:\; 0=\mathcal{T}'_0\subsetneq\mathcal{T}'_1\subsetneq\cdots\subsetneq\mathcal{T}'_{n-1}
\subsetneq\mathcal{T}'_{n}=\mathcal{D},
$
constructed above is called the \emph{right mutation} of $\mathscr{T}$ at $\mathcal{T}_{i}$ for $1\leq i\leq n-1$.
Similarly, one can define \emph{left mutation} $\check{\sigma}_i$.
\end{defn}

Recall from \cite{bonk} that a finite filtration $\mathscr{T}$ is called \emph{$\infty$-admissible} if all the iterated right and left mutations of $\mathscr{T}$ are admissible.
Inspired by the relationship with admissible SODs in Proposition \ref{finit t-stab and finite right admiss}, we say that a finite filtration $\mathscr{T}$ is \emph{$\infty$-strongly admissible} (\emph{$\infty$-s-admissible}) if all the iterated right and left mutations of $\mathscr{T}$ are s-admissible.

For a finite admissible filtration $\mathscr{T}$, we have obviously that $\sigma_i\check{\sigma}_i\mathscr{T}=\mathscr{T}=
\check{\sigma}_i\sigma_i\mathscr{T}$.
Moreover,
$\sigma_i$ preserves finite $\infty$-s-admissible filtrations in $\mathcal{D}$ by \cite{bonk}.

\subsection{A bijection and compatibility}
We define a natural partial order on admissible filtrations as follows.
\begin{defn}
Let
\begin{equation*}
\mathscr{X}:\; 0=\mathcal{X}_0\subsetneq\mathcal{X}_1\subsetneq\cdots\subsetneq\mathcal{X}_{n-1}
\subsetneq\mathcal{X}_{n}=\mathcal{D},
\end{equation*}
\begin{equation*}
\mathscr{Y}:\; 0=\mathcal{Y}_0\subsetneq\mathcal{Y}_1\subsetneq\cdots\subsetneq\mathcal{Y}_{m-1}
\subsetneq\mathcal{Y}_{m}=\mathcal{D},
\end{equation*} be
two finite left (right) admissible filtrations.
They are called \emph{equivalent} if there exists an order-preserved bijective map $r:\Phi_n\rightarrow\Psi_m$ such that  $\mathcal{Y}_{r(i)}=\mathcal{X}_i$ for all $i\in\Phi_n$; $\mathscr{X}$ is called \emph{finer} than $\mathscr{Y}$, if there exists a surjective map $r:\Phi_n\rightarrow\Phi_m$ such that
\begin{itemize}
 \item [(1)] $i'>i''$ implies $ r(i')\geq r(i'')$;
  \item [(2)] for any $j\in\Phi_m$, $\mathcal{Y}_{j}=\langle \mathcal{X}_{i}\mid r(i)= j\rangle$.
\end{itemize}
\end{defn}
Minimal elements with respect to this partial ordering is called
the \emph{finest} left (right) admissible filtrations.

We are in a position to show the following bijection.

\begin{prop}\label{mutaion preserve finest}
Keep notations as above. Then $\xi$ restricts to a bijection
$$\begin{array}{lcr}
\xi:\left\{
\mbox{finest $\infty$-admissible SODs of $\mathcal{D}$}
\right\}&
\stackrel{1:1}{\longleftrightarrow}&
\left\{ \mbox{finite finest $\infty$-s-admissible filtrations in $\mathcal{D}$} \right\},
\end{array}$$
which is compatible with mutations on both sides.
\end{prop}

\begin{proof}
Note that the bijection $\xi$ in
\ref{finit t-stab and finite right admiss} is order-preserving.
Hence, it suffices to show that $\xi$ is compatible with mutations.
We treat only right mutations, the left case being similar.
Let $(\Pi_i;i\in\Phi_n)$ be an $\infty$-admissible SOD of $\mathcal{D}$.
Fix  $1\leq i\leq n-1$, we will prove that $\xi\rho_{i}=\sigma_{n-i}\xi$.

Applying the right mutation to $(\Pi_i;i\in\Phi_n)$ at $\Pi_{i}$
gives an $\infty$-admissible SOD
$(\Pi'_i;i\in\Phi_n)$, where
$\Pi'_{i}=\Pi^{\bot}_{i}\cap\langle \Pi_{i}, \Pi_{i+1}\rangle$, $\Pi'_{i+1}=\Pi_{i}$, and $\Pi'_{j}=\Pi_{j}$ for $j\in\Phi_n\setminus\{i, i+1\}$.
Then, $\xi(\Pi'_i;i\in\Phi_n)=\mathscr{T}:\;
0=\mathcal{T}_0\subsetneq\mathcal{T}_1\subsetneq\cdots\subsetneq\mathcal{T}_{n-1}
\subsetneq\mathcal{T}_{n}=\mathcal{D}$
is  a finite s-admissible filtration,
where
$\mathcal{T}_{n-i}=\langle \Pi_{i},\Pi_{\geq i+2}\rangle$,
$\mathcal{T}_{n-i+1}=\langle \Pi^{\bot}_{i}\cap\langle \Pi_{i}, \Pi_{i+1}\rangle, \Pi_{i}, \Pi_{\geq i+2}\rangle=\Pi_{\geq i}$ since $\langle \Pi^{\bot}_{i}\cap\langle \Pi_{i}, \Pi_{i+1}\rangle, \Pi_{i}\rangle=\langle \Pi_{i}, \Pi_{i+1}\rangle$, and $\mathcal{T}_{j}=\Pi_{\geq n-j+1}$ for  $j\in\Phi_n\setminus\{n-i, n-i+1\}$.
Hence, $\xi\rho_{i}(\Pi_i;i\in\Phi_n)=\mathscr{T}$.

On the other hand, $\xi(\Pi_i;i\in\Phi_n)=
\mathscr{T}':\;
0=\mathcal{T}'_0\subsetneq\mathcal{T}'_1\subsetneq\cdots\subsetneq\mathcal{T}'_{n-1}
\subsetneq\mathcal{T}'_{n}=\mathcal{D}$
is another finite s-admissible filtration,
where
$\mathcal{T}'_{i}=\Pi_{\geq n-i+1}\;\,\text{for}\;\,j\in\Phi_n$.
Now, applying the right mutation to $\mathscr{T}'$ at $\mathcal{T}'_{n-i}$ gives a filtration
$\mathscr{F}:\;
0=\mathcal{F}_0\subsetneq\mathcal{F}_1\subsetneq\cdots
\subsetneq\mathcal{F}_{n-1}
\subsetneq\mathcal{F}_{n}=\mathcal{D}$,
where $\mathcal{F}_{n-i}=\langle
\mathcal{T}^{\bot}_{n-i}\cap\mathcal{T}_{n-i+1},
\mathcal{T}_{n-i-1}\rangle$,
and $\mathcal{F}_{j}=\Pi_{\geq n-j+1}$ for  $j\in\Phi_n\setminus\{n-i\}$.
Therefore, $\sigma_{n-i}\xi(\Pi_i;i\in\Phi_n)=\mathscr{F}$.

Noticing that $\mathcal{T}^{\bot}_{n-i}\cap\mathcal{T}_{n-i+1}=
\Pi^{\bot}_{\geq i+1}\cap\Pi_{\geq i}=\Pi_{i}$, we obtain
$\mathcal{F}_{n-i}=\langle \Pi_{i},\mathcal{T}_{n-i-1}\rangle=\langle \Pi_{i},\Pi_{\geq i+2}\rangle$.
This gives $\mathcal{F}_{j}=\mathcal{T}_{j}$ for all $j\in\Phi_n$ and hence $\mathscr{F}=\mathscr{T}$,
which implies that
$\mathscr{F}$ is finite s-admissible.
Thus, $\sigma_{n-i}\xi(\Pi_i;i\in\Phi_n)=
\xi\rho_{i}(\Pi_i;i\in\Phi_n)$.
\end{proof}

The following braid relations hold for the operators $\rho_i$ on the set of finest $\infty$-admissible SODs.
\begin{cor}
The operators $\rho_i$ satisfy the  braid relations:
$$\begin{array}{lcr}
\mbox{$\rho_i\rho_j=\rho_j\rho_i$ for $\lvert i-j\rvert\geq2$, and $\rho_i\rho_{i+1}\rho_i=\rho_{i+1}\rho_i\rho_{i+1}$.}
\end{array}$$
\end{cor}

\begin{proof}
The first relation is straightforward.
For the second one, the braid relations of $\sigma_i$ in \cite[Prop. 4.9]{bonk} imply
\begin{equation*}
    \xi(\rho_i\rho_{i+1}\rho_i)
    = (\sigma_{n-i}\sigma_{n-i-1}\sigma_{n-i})\xi =(\sigma_{n-i-1}\sigma_{n-i}\sigma_{n-i-1})\xi
     = \xi(\rho_{i+1}\rho_{i}\rho_{i+1}).
\end{equation*}
The bijectivity of $\xi$ therefore gives  $\rho_i\rho_{i+1}\rho_i=\rho_{i+1}\rho_{i}\rho_{i+1}$.
\end{proof}

\section{SODs and exceptional sequences}  \label{Sect 5}

In this section we classify SODs in terms of exceptional sequences in triangulated categories admitting a Serre functor, with direct applications to the derived categories of the projective plane, weighted projective lines, and finite acyclic quivers.

\subsection{Triangulated category with a Serre functor}
Throughout this section, the triangulated category $\mathcal{D}$ is assumed to have a Serre functor $\mathbb{S}$.

Recall from \cite{bonk} that a subcategory $\mathcal{A}\subset\mathcal{D}$ is called \emph{$\infty$-admissible} if all the iterated right
and left orthogonals to $\mathcal{A}$ are admissible.
Then we have the following results.

\begin{lem} [\cite{bond,bonk}] \label{right addmiss and sod2}
Any left  (right) admissible subcategory in $\mathcal{D}$ is $\infty$-admissible.
Moreover, if $\mathcal{A}\subset\mathcal{D}$ is right admissible with the inclusion functor $i$, then  $\mathcal{A}$ admits a Serre functor
$\mathbb{S}_{\mathcal{A}}=i^{!}\circ\mathbb{S}\circ i$.
\end{lem}

\begin{lem} \label{infinity admissible for two subcats}
Assume that $0\subsetneq\mathcal{T}_1\subsetneq\mathcal{T}_2\subsetneq\mathcal{D}$
is a finite right admissible filtration in  $\mathcal{D}$.
Then $\mathcal{T}^{\bot}_{1}\cap\mathcal{T}_{2}$ is admissible in $\mathcal{D}$.
\end{lem}

\begin{proof}
By Lemma \ref{right addmiss and sod2}, both $\mathcal{T}_{1}$ and $\mathcal{T}_{2}$ admit a Serre functor, and $\mathcal{T}_{1}$ is $\infty$-admissible in $\mathcal{T}_{2}$.
It follows that $\mathcal{T}^{\bot}_{1}\cap\mathcal{T}_{2}$ is admissible in $\mathcal{T}_{2}$.
By transitivity in \cite[Lem. 1.11]{bonk}, $\mathcal{T}^{\bot}_{1}\cap\mathcal{T}_{2}$ is admissible in $\mathcal{D}$.
\end{proof}

\begin{lem} \label{the three sets are equal under Serre}
Assume that
$
\mathscr{T}:\; 0=\mathcal{T}_0\subsetneq\mathcal{T}_1\subsetneq\cdots\subsetneq\mathcal{T}_{n-1}
\subsetneq\mathcal{T}_{n}=\mathcal{D}
$
is a finite filtration in $\mathcal{D}$.
Then $\mathscr{T}$ is left (right) admissible if and only if it is $\infty$-s-admissible.
\end{lem}

\begin{proof}
The proof of sufficiency is straightforward.
To prove necessity,  let $\mathscr{T}$ be a left (right) admissible filtration.
It suffices to show that  $\sigma_i\mathscr{T}$ is s-admissible for any $1\leq i\leq n-1$.

By Lemma \ref{infinity admissible for two subcats}, each $\mathcal{T}^{\bot}_{i}\cap\mathcal{T}_{i+1}$ is admissible in $\mathcal{D}$.
Thus, $\mathscr{T}$ is s-admissible.
Let $\mathcal{T}'_{i}=\langle \mathcal{T}^{\bot}_{i}
\cap\mathcal{T}_{i+1},\mathcal{T}_{i-1}\rangle$.
Observe that $\mathcal{T}^{\bot}_{i}
\cap\mathcal{T}_{i+1}\subseteq\mathcal{T}^{\bot}_{i-1}$
since $\mathcal{T}_{i-1}\subsetneq\mathcal{T}_{i}$.
Thus, the pair $(\mathcal{T}^{\bot}_{i}
\cap\mathcal{T}_{i+1},\mathcal{T}_{i-1})$ forms a SOD of $\mathcal{T}'_{i}$.
This implies $\mathcal{T}^{\bot}_{i-1}\cap\mathcal{T}'_{i}=
\mathcal{T}^{\bot}_{i}
\cap\mathcal{T}_{i+1}$.
It follows from \cite[Prop. 1.12]{bonk} that $\langle \mathcal{T}^{\bot}_{i}
\cap\mathcal{T}_{i+1},\mathcal{T}_{i-1}\rangle$ is admissible in $\mathcal{D}$.
Therefore, $\sigma_i\mathscr{T}$ is admissible and hence s-admissible.
Analogously, one can prove that $\check{\sigma}_i\mathscr{T}$ is s-admissible.
Consequently, $\mathscr{T}$ is $\infty$-s-admissible.
\end{proof}

\begin{lem} \label{admissible sucat and sod gen by excep seq}
Any SOD of $\mathcal{D}$ is $\infty$-admissible.
\end{lem}

\begin{proof}
Let $(\Pi_{i};i\in\Phi_n)$ be a SOD of $\mathcal{D}$.
By Proposition \ref{finit t-stab and finite right admiss}, there exists a
finite right admissible filtration
$
\mathscr{T}:\;
0=\mathcal{T}_0\subsetneq\mathcal{T}_1\subsetneq\cdots\subsetneq\mathcal{T}_{n-1}
\subsetneq\mathcal{T}_{n}=\mathcal{D},\;\,\text{where}\;\,
\mathcal{T}_{i}=\Pi_{\geq n-i+1}\;\,\text{for}\;\,i\in\Phi_n
$.
By Lemma \ref{the three sets are equal under Serre}, $\mathscr{T}$ is  $\infty$-s-admissible.
Hence, by Proposition \ref{mutaion preserve finest}, $(\Pi_{i};i\in\Phi_n)$ is $\infty$-admissible.
\end{proof}

\subsection{Exceptional sequences}
Recall that an object $E$ in $\mathcal{D}$ is called \emph{exceptional} if
$$\Hom(E,E[l])= \left\{\begin{array}{l} \mathbf{k},\quad l=0,\\
0,\quad  l\neq0.\\
 \end{array}\right.$$
An ordered sequence $(E_{i};i\in\Phi_n):=(E_{1},E_{2},\cdots,E_{n})$ in $\mathcal{D}$ is
called an \emph{exceptional sequence}  if each $E_i$ is exceptional and $\Hom(E_j,E_i[l])=0$ for all $i<j$ and $l\in\mathbb{Z}$. It is said to be \emph{full} if  $\langle E_{1},E_{2},\cdots,E_{n}\rangle=\mathcal{D}$.

The left mutation of an exceptional pair $\mathfrak{E}=(E,F)$ is the pair
$\mathrm{L}_{E}\mathfrak{E}=(L_{E}F,E)$, where $L_{E}F$ is defined by the triangle
$L_{E}F\rightarrow\Hom(E,F)\otimes E \rightarrow F\rightarrow L_{E}F[1]$.
The \emph{left mutation} $\mathrm{L}_{i}$ of an exceptional sequence $\mathfrak{E}=(E_{i};i\in\Phi_n)$ at $E_{i}$ is defined as the left mutation
of a pair of adjacent objects in this sequence:
$$\mathrm{L}_{i}\mathfrak{E}=(E_{1},E_{2},\cdots,E_{i-1},
L_{E_i}E_{i+1},E_i,E_{i+2},\cdots,E_{n}).$$
Similarly, one defines the \emph{right  mutation} $\mathrm{R}_{i}$ of an exceptional collection $\mathfrak{E}=(E_{i};i\in\Phi_n)$.

It is  known that an exceptional sequence generates an admissible subcategory, and a full exceptional sequence naturally induces a SOD, cf. \cite{bond}.
Let $(E_i;i \in \Phi_n)$ and $(F_i;i \in \Phi_n)$ be two exceptional sequences in  $\mathcal{D}$.
We say that  $(E_i;i \in \Phi_n)$ is \emph{equivalent}  to $(F_i;i \in \Phi_n)$ if $\langle E_i\rangle=\langle F_i\rangle$ for all $i \in \Phi_n$.

In the rest of this section, we impose the following condition.
\begin{assump}\label{Assump}
Every admissible subcategory of $\mathcal{D}$ is generated by an exceptional sequence.
\end{assump}
This enables us to obtain the following result.

\begin{thm} \label{finest sod and full exceptional sequence}
Keep notations as above.
There is a bijection
\begin{eqnarray*}
 \chi:\{\mbox{equivalence classes  of full  exceptional sequences in $\mathcal{D}$}\}&\longrightarrow&
\{\mbox{finest SODs  of $\mathcal{D}$}\}, \\
(E_1, E_2,\cdots,  E_n) &\longmapsto&
(\langle E_1\rangle, \langle E_2\rangle,\cdots, \langle E_n\rangle)
\end{eqnarray*}
which is compatible with mutations in the sense that
\begin{eqnarray*}
\chi\mathrm{L}_i=\rho_i\chi & \mbox{and} &\chi\mathrm{R}_i=\check{\rho}_i\chi.
\end{eqnarray*}
\end{thm}

\begin{proof}
Let $(E_i;i \in \Phi_n)$ be a full exceptional sequence in $\mathcal{D}$.
We need to show  that the induced SOD $(E_i;i \in \Phi_n)$ with $\Pi_{i}=\langle E_i\rangle$ is finest.
Since $\Pi_{i}=\bigvee_{m\in\mathbb{Z}}\add\{E_i[m]\}$, it follows that
$$\Hom(\langle X\rangle,\langle Y\rangle)\supseteq
\Hom^{\bullet}(E_i,E_i)\subseteq
\Hom(\langle Y\rangle,\langle X\rangle)\;\, \text{for all}\;\, 0\neq X,Y\in\Pi_{i}.$$
Thus, by Proposition \ref{necess cond for finest SOD}, $(\Pi_{i};i\in\Phi_n)$ is finest.
This shows that $\chi$ is well-defined.

Conversely, let $(\Pi_{i};i\in\Phi_n)$ be a finest  SOD of $\mathcal{D}$.
By Lemma \ref{admissible sucat and sod gen by excep seq} and assumption, each $\Pi_i$ is admissible and generated by an exceptional sequence $(E_{\psi};\psi\in I_{i})$ with $I_{i}$ a linearly ordered set.

We claim $\lvert I_{i}\rvert=1$ for all $i\in\Phi_n$.
Otherwise, suppose that there exists $\lvert I_{i}\rvert\geq2$ for some $i$.
Let $P_{\psi}=\langle E_{\psi}\rangle$ for $\psi\in I_{i}$.
Then $\Pi_{i}=\langle P_{\psi}\mid\psi\in I_{i} \rangle$.
It follows that $(I_{i},\{P_{{\psi}}\}_{\psi \in I_{i}})$ is a locally finite t-stability on $\Pi_{i}$.
Let $\Psi=(\Phi_n\setminus\{i\}) \cup \{I_{i}\}$, which is a linearly ordered set with the relations $j'>\psi'>\psi''>j''$ for any $j'>i>j'' \in \Phi_n$ and any $\psi'>\psi''\in I_{i}$.
By Proposition \ref{local finite refinement construction},
$(\Psi,\{P_{\psi}\}_{\psi \in \Psi})$ is a finite t-stability on $\mathcal{D}$.
Therefore, we get a SOD $(P_{\psi};\psi \in \Psi)$  which is finer than $(\Pi_{i};i\in\Phi_n)$, a contradiction.
This proves the claim, implying that $(E_i;i \in \Phi_n)$ is a full exceptional sequence. A simple verification shows that this defines the inverse of $\chi$.

For compatibility of $\chi$ with mutations, we prove only the first equality; the other is similar.
Let $\mathfrak{E}=(E_j;j \in \Phi_n)$ be a full exceptional sequence in $\mathcal{D}$.
For a fixed index $1\leq i\leq n-1$, let $\mathrm{L}_i\mathfrak{E}=\mathfrak{E}'=(E'_j;j \in \Phi_n)$  be the left mutation of $\mathfrak{E}$ at $i$, where $E'_{j}=E_{j}$ for $j\in\Phi_n\setminus\{i,i+1\}$, $E'_{i+1}=E_{i}$,
and $E'_{i}$ is defined by the triangle
$E'_{i}\rightarrow\Hom(E_{i},E_{i+1})\otimes E_{i}
\rightarrow{}E_{i+1}\rightarrow E'_{i}[1]$.
Then $\chi\mathrm{L}_i\mathfrak{E}=(\Pi'_{j};j\in\Phi_n)$
is a finest SOD with $\Pi'_{j}=\langle E'_j\rangle$.

On the other hand,  $\chi\mathfrak{E}=(\Pi_{j};j\in\Phi_n)$ is a finest SOD
with $\Pi_{j}=\langle E_j\rangle$.
Its right mutation at $i$ yields another finest SOD
$\rho_i(\Pi_{j};j\in\Phi_n)=(\Pi''_{j};j\in\Phi_n)$, where $\Pi''_{i}=\Pi^{\bot}_{i}\cap\langle \Pi_{i}, \Pi_{i+1}\rangle,\;\, \Pi''_{i+1}=\Pi_{i},\;\,\text{and}\;\, \Pi''_{j}=\Pi_{j}\;\,\text{for}\;\,j\in\Phi_n\setminus\{i, i+1\}$.
Thus,  $\rho_i\chi\mathfrak{E}=(\Pi''_{j};j\in\Phi_n)$.

To prove $(\Pi'_{j};j\in\Phi_n)=(\Pi''_{j};j\in\Phi_n)$, note that $\Pi'_{j}=\Pi''_{j}$  for all $j\in\Phi_n\setminus\{i\}$.
Moreover,
\begin{equation*}
\begin{split}
   \Pi''_{i} &={}^{\bot}\langle \Pi_j\mid j\leq i-1\rangle\cap
\langle \Pi_i,\Pi_j\mid j\geq i+2\rangle^{\bot}={}^{\bot}\langle E_j\mid j\leq i-1\rangle\cap
\langle E_i,E_j\mid j\geq i+2\rangle^{\bot}.\\
\end{split}
\end{equation*}
Therefore, $\Pi''_{i}=\langle E'_i\rangle=\Pi'_{i}$, and we conclude that $\chi\mathrm{L}_i\mathfrak{E}=\rho_i\chi\mathfrak{E}$.
\end{proof}

\begin{cor}
Each finite t-stability on  $\mathcal{D}$ can be refined to a finite finest one.
\end{cor}

\begin{proof}
Let $(\Phi_m,\{\Pi_{i}\}_{i\in\Phi_m})$ be a finite t-stability, and thus $(\Pi_{i};i\in\Phi_m)$ is a SOD of $\mathcal{D}$.
Note that each $\Pi_{i}$ is generated by an exceptional sequence $ (E_{\psi};\psi\in I_{i})$.
This gives a decomposition $\Phi_n:=I_1\cup I_2\cup \cdots \cup I_m$ and a
full exceptional sequence $(E_{\psi};\psi \in \Phi_n)$.
Let $P_{\psi}=\langle E_{\psi}\rangle$ for each $\psi\in I_{i}$.
By Theorem \ref{finest sod and full exceptional sequence} and Proposition  \ref{t-stab and sod}, this gives a finite finest t-stability  $(\Phi_n,\{P_{\psi}\}_{\psi\in\Phi_n})$ on $\mathcal{D}$.

Let $r:\, \Phi_n \rightarrow \Phi_m$ be such that $r(\psi)=i$ for all $\psi\in I_i$.
From the decomposition of $\Phi_n$, we know that  $r(\psi')<r(\psi'')$ for $\psi'<\psi''\in\Phi_n$.
On the other hand, for any $i\in\Phi_m$,
$\Pi_{i}=\langle E_{\psi}\mid\psi\in I_{i}\rangle=
\langle E_{\psi}\mid\psi\in r^{-1}(i)\rangle$.
It follows that $(\Phi_n,\{P_{\psi}\}_{\psi\in\Phi_n})
\preceq(\Phi_m,\{\Pi_{i}\}_{i\in\Phi_m})$.
\end{proof}

\begin{rem} \label{surj cond}
The map $\chi$ in Theorem \ref{finest sod and full exceptional sequence} is injective for any triangulated category $\mathcal{D}$.

Now assume that each finite t-stability (and hence any SOD) on $\mathcal{D}$ can be refined to a finite finest one.
Observe that each admissible subcategory  $\mathcal{A}$ of  $\mathcal{D}$ fits into a SOD of the form  $(\mathcal{A}^{\bot},\mathcal{A})$ or $(\mathcal{A},{}^{\bot}\mathcal{A})$, which is necessarily coarser than a finest one.
Consequently, the surjectivity of $\chi$ implies that the assumption in Theorem \ref{finest sod and full exceptional sequence} must be satisfied, which holds for multiple important categories arising in algebraic geometry and representation theory, including (up to triangle equivalence) the derived categories of
\begin{itemize}
  \item [(1)] the coherent sheaves category over  the projective plane, cf. \cite[Thm. 4.2]{pir};
  \item [(2)] the coherent sheaves category over a weighted projective line of any type,
      cf. \cite[Cor. 8.7]{thic};
  \item [(3)] category of representations of any finite acyclic quiver, cf. \cite[Cor. 3.7]{robo}.
\end{itemize}

Let $(E_i;i \in \Phi_n)$ be a full exceptional sequence in $\mathcal{D}$, and let $\Phi_n=I_1\cup I_2\cup \cdots \cup I_m$ be a decomposition of $\Phi_n$ such that $\psi<\psi'$ whenever $\psi\in I_{i}$ and $\psi'\in I_{j}$ with $i<j$.
Set $\mathcal{E}_j=(E_{\psi};\psi\in I_{j})$. Then we say that $(\mathcal{E}_j;j \in \Phi_m)$ is a \emph{partition} of the full exceptional sequence $(E_i;i \in \Phi_n)$.

As a conclusion of Proposition \ref{any sod construction}, we obtain a classification of SODs  for these categories 
in terms of partitions of full exceptional sequences. In particular, the finest SODs are in one-to-one correspondence with  full exceptional sequences.
More precisely,
we have the following bijection:
\begin{eqnarray*}
 \chi^{-1}:\{\mbox{SODs  of $\mathcal{D}$}\}&\longrightarrow&
\{\mbox{partitions of equivalence classes of  full exceptional sequences in $\mathcal{D}$}\}.\\
(\langle \mathcal{E}_1\rangle, \langle \mathcal{E}_2\rangle,\cdots, \langle \mathcal{E}_m\rangle)&\longmapsto&
 (\mathcal{E}_1, \mathcal{E}_2,\cdots  \mathcal{E}_m)
\end{eqnarray*}
\end{rem}

\section{Reduction of SODs and their mutation graphs}

In this section we investigate SODs using a reduction method analogous to that for thick subcategories, and show a criterion for the connectedness of their mutation graphs.

\subsection{SOD reduction}
We start with the following collection of reductions of thick subcategories.
Let $\mathcal{U} \subseteq \mathcal{D}$ be a thick subcategory, $Q_{\mathcal{U}} : \mathcal{D} \rightarrow
\mathcal{D} / \mathcal{U}$ the quotient functor, where $\mathcal{D} / \mathcal{U}$ is the Verdier quotient.
From now onward, we abbreviate $Q_{\mathcal{U}}$ as $Q$.
Recall from \cite[Prop. II.2.3.1]{ver} that there is a bijection
\[
  \bigg\{\begin{array}{l}
      \mbox{triangulated subcategories $\mathcal{A}\subseteq\mathcal{D}$ } \\[2mm]
      \mbox{with $\mathcal{U}\subseteq\mathcal{A}\subseteq\mathcal{D}$}
    \end{array}\bigg\}
  \stackrel{1:1}{\longleftrightarrow}
  \bigg\{\begin{array}{l}
  \mbox{triangulated  subcategories of $\mathcal{D}/\mathcal{U}$}
  \end{array}\bigg\}
\]
taking $\mathcal{A} \mapsto Q\mathcal{A}$.
Under the bijection, thick subcategories correspond to thick subcategories.
Further, recall from \cite[Thm. C]{joka} that
there is a bijection
\[
  \bigg\{(\mathcal{X},\mathcal{Y} )
  \bigg|
    \begin{array}{l}
      \mbox{$\mathcal{X}, \mathcal{Y} \subseteq \mathcal{D}$ thick subcategories } \\[2mm]
      \mbox{with $\mathcal{X }\cap\mathcal{Y} = \mathcal{U}$ and $\mathcal{X} * \mathcal{Y} = \mathcal{D}$ }
    \end{array}\bigg\}
  \stackrel{1:1}{\longleftrightarrow}
  \bigg\{\begin{array}{l}
  \mbox{stable $t$-structures in $\mathcal{D}/\mathcal{U}$}
  \end{array}\bigg\}
\]
sending $(\mathcal{X},\mathcal{Y}) \mapsto (Q\mathcal{X},Q\mathcal{Y})$, where $\mathcal{X} * \mathcal{Y}$ is defined by
\[
  \mathcal{X} * \mathcal{Y}
  = \bigg\{Z \in \mathcal{D} \bigg|
      \begin{array}{ll}
         \mbox{there is a  triangle $X \rightarrow Z \rightarrow Y\rightarrow X[1]$ in $\mathcal{D}$
                with $X \in \mathcal{X}$, $Y \in \mathcal{Y}$}
      \end{array}
    \bigg\}.
\]
Note that $(\mathcal{X},\mathcal{Y})$ is a stable $t$-structure if and only if $(\mathcal{Y},\mathcal{X})$ is a SOD.
In this situation,  $\mathcal{X}$ is right admissible and $\mathcal{Y}$ is left admissible (both being thick subcategories), with natural equivalences $\mathcal{X}^{\bot}\simeq \mathcal{D}/\mathcal{X}$ and ${}^{\bot}\mathcal{Y}\simeq \mathcal{D}/\mathcal{Y}$, cf. \cite[Prop. 1.6]{bonk}.

\begin{lem} \label{admiss sub bijection}
For a (left) right admissible $\mathcal{U} \subsetneq \mathcal{D}$ with quotient $Q: \mathcal{D} \to \mathcal{D}/\mathcal{U}$, there is a bijection
\[
  \bigg\{\begin{array}{l}
      \mbox{(left) right admissible subcategories $\mathcal{A}$ of $\mathcal{D}$ } \\[2mm]
      \mbox{with $\mathcal{U}\subsetneq\mathcal{A}\subsetneq\mathcal{D}$}
    \end{array}\bigg\}
  \stackrel{1:1}{\longleftrightarrow}
  \bigg\{\begin{array}{l}
  \mbox{(left) right admissible subcategories of $\mathcal{D}/\mathcal{U}$}
  \end{array}\bigg\}
\]
{\em taking $\mathcal{A} \mapsto Q\mathcal{A}$. }
\end{lem}

\begin{proof}
We prove the right admissible case; the left admissible case is analogous.
For a right admissible subcategory $\mathcal{A}$ of $\mathcal{D}$ with $\mathcal{U} \subsetneq \mathcal{A} \subsetneq \mathcal{D}$, it suffices to show that $Q\mathcal{A}$ is right admissible.

We first show that $\mathcal{A}\cap\langle\mathcal{A}^{\bot},\mathcal{U}\rangle
=\mathcal{U}$.
Indeed, the right inclusion ``$\supseteq$'' is clear.
For the left inclusion ``$\subseteq$'', let $0\neq Z\in\mathcal{A}\cap
\langle\mathcal{A}^{\bot},\mathcal{U}\rangle$.
Note that the right admissibility of $\mathcal{U}$ in $\mathcal{D}$ implies  its admissibility in both $\mathcal{A}$ and $\langle\mathcal{A}^{\bot},\mathcal{U}\rangle$, and that $(\mathcal{A}^{\bot},\mathcal{U})$ is a SOD of the latter category.
This gives two triangles
$$U_1\rightarrow Z\rightarrow V_1\rightarrow U_1[1]\quad\text{and} \quad
U_2\rightarrow Z\rightarrow V_2\rightarrow U_2[1]$$
with $U_1\in\mathcal{U},V_1\in\mathcal{U}^{\bot}\cap\mathcal{A}
\subseteq\mathcal{U}^{\bot}$ and $U_2\in\mathcal{U},V_2\in\mathcal{A}^{\bot}\subseteq\mathcal{U}^{\bot}$.
Since the decomposition triangle for $Z$ under the SOD $(\mathcal{U}^{\bot},\mathcal{U})$ is unique, it follows that $U_1\cong U_2$ and $V_1\cong V_2$.
The fact $\mathcal{A}\cap\mathcal{A}^{\bot}=0$ implies that $V_1\cong V_2=0$ and hence $Z\in\mathcal{U}$.
This proves the left inclusion.

Moreover, we have $\mathcal{A}*\langle\mathcal{A}^{\bot},\mathcal{U}\rangle
=\mathcal{A}*\mathcal{A}^{\bot}=\mathcal{D}$.
Consequently, by \cite[Thm. C]{joka}, $(Q\mathcal{A},Q\langle\mathcal{A}^{\bot},\mathcal{U}\rangle)$
is a stable $t$-structure in $\mathcal{D} / \mathcal{U}$, implying that  $Q\mathcal{A}$ is right admissible.
\end{proof}

Recall from \eqref{admissble subcats seq} that a finite admissible filtration $\mathscr{T}$ has the form
$0=\mathcal{T}_0\subsetneq\mathcal{T}_1\subsetneq\cdots\subsetneq\mathcal{T}_{n-1}
\subsetneq\mathcal{T}_{n}=\mathcal{D}$.

\begin{prop}  \label{admiss filt bijection}
For an admissible $\mathcal{U} \subsetneq \mathcal{D}$ with quotient $Q: \mathcal{D} \to \mathcal{D}/\mathcal{U}$, there is a bijection
\[
  \bigg\{\begin{array}{l}
      \mbox{finite finest $\infty$-s-admissible filtrations $\mathscr{T}$ of $\mathcal{D}$} \\[2mm]
      \mbox{with $\mathcal{U}\subsetneq\mathcal{T}_1\subsetneq\mathcal{D}$}
    \end{array}\bigg\}
  \stackrel{1:1}{\longleftrightarrow}
  \bigg\{\begin{array}{l}
  \mbox{finite finest $\infty$-s-admissible filtrations  $\mathcal{D}/\mathcal{U}$}
  \end{array}\bigg\}
\]
{\em taking $\mathscr{T} \mapsto Q\mathscr{T}$, which is compatible with mutations. }
\end{prop}

\begin{proof}
Let $\mathscr{T}:\;
0=\mathcal{T}_0\subsetneq\mathcal{T}_1\subsetneq\cdots\subsetneq\mathcal{T}_{n-1}
\subsetneq\mathcal{T}_{n}=\mathcal{D}$ be a finite  $\infty$-s-admissible filtration in $\mathcal{D}$.
We first show that $Q\mathscr{T}$ is a finite s-admissible s-filtration.
It suffices to show that each $(Q\mathcal{T}_i)^{\bot}\cap Q\mathcal{T}_{i+1}$ is admissible in $\mathcal{D}/\mathcal{U}$.
The stable t-struture $(Q\mathcal{T}_i,Q\langle\mathcal{T}^{\bot}_i,\mathcal{U}\rangle)$ ensures that $(Q\mathcal{T}_i)^{\bot}=
Q\langle\mathcal{T}^{\bot}_i,\mathcal{U}\rangle$.
By \cite[Lem. 2.4 (i.a)]{joka},
$$
(Q\mathcal{T}_i)^{\bot}\cap Q\mathcal{T}_{i+1}=
Q\langle\mathcal{T}^{\bot}_i,\mathcal{U}\rangle\cap Q\mathcal{T}_{i+1}=
Q(\langle\mathcal{T}^{\bot}_i,\mathcal{U}\rangle
\cap\mathcal{T}_{i+1})=
Q(\langle\mathcal{T}^{\bot}_i\cap\mathcal{T}_{i+1},
\mathcal{U}\rangle).
$$
Thus, by Lemma \ref{admiss sub bijection},  the admissibility of $\langle\mathcal{T}^{\bot}_i\cap\mathcal{T}_{i+1},
\mathcal{U}\rangle$ implies that of $(Q\mathcal{T}_i)^{\bot}\cap Q\mathcal{T}_{i+1}$.
This shows the desired filtration.

Now, applying the right mutation to $\mathscr{T}$ at $\mathcal{T}_{i}$ gives a finite  filtration
$
\mathscr{T}':\;
0=\mathcal{T}'_0\subsetneq\mathcal{T}'_1\subsetneq\cdots
\subsetneq\mathcal{T}'_{n-1}
\subsetneq\mathcal{T}'_{n}=\mathcal{D},
\;\,\text{where}\;\, \mathcal{T}'_{i}=\langle
\mathcal{T}^{\bot}_{i}\cap\mathcal{T}_{i+1},
\mathcal{T}_{i-1}\rangle,
$
and $\mathcal{T}'_{j}=\mathcal{T}_{j}$ for  $j\in\Phi_n\setminus\{i\}$.

On the other hand, applying the right mutation to $Q\mathscr{T}$ at $Q\mathcal{T}_{i}$ yields a finite  filtration
$(Q\mathscr{T})':\;
0=(Q\mathcal{T})'_0\subsetneq (Q\mathcal{T})'_1\subsetneq\cdots
\subsetneq (Q\mathcal{T})'_{n-1}
\subsetneq (Q\mathcal{T})'_{n}=\mathcal{D}$,
where
$(Q\mathcal{T})'_{i}=\langle
(Q\mathcal{T}_{i})^{\bot}\cap Q\mathcal{T}_{i+1},
Q\mathcal{T}_{i-1}\rangle$,
and $(Q\mathcal{T})'_{j}=Q\mathcal{T}_{j}$ for  $j\in\Phi_n\setminus\{i\}$.
Since $(Q\mathcal{T}_i)^{\bot}\cap Q\mathcal{T}_{i+1}=
Q(\langle\mathcal{T}^{\bot}_i\cap\mathcal{T}_{i+1},
\mathcal{U}\rangle)$, it follows that
$$Q\mathcal{T}'_{i}=Q\langle
\mathcal{T}^{\bot}_{i}\cap\mathcal{T}_{i+1},
\mathcal{T}_{i-1}\rangle=\langle
(Q\mathcal{T}_{i})^{\bot}\cap Q\mathcal{T}_{i+1},
Q\mathcal{T}_{i-1}\rangle=(Q\mathcal{T})'_{i}.$$
This gives that $Q\sigma_i\mathscr{T}=\sigma_iQ\mathscr{T}$. Analogously, one can show that $Q\check{\sigma}_i\mathscr{T}=\check{\sigma}_iQ\mathscr{T}$. Thus we obtain the order-preserving bijection.
\end{proof}

The following result shows a correspondence for the reduction of SODs.

\begin{prop} \label{reduction thm}
Keep notation as in Proposition \ref{admiss filt bijection}. There is a bijection
\[
  \bigg\{
    \begin{array}{l}
      \mbox{finest $\infty$-admissible SODs $(\Pi_{i};i\in\Phi_n)$ of  $\mathcal{D}$} \\[2mm]
      \mbox{with $\mathcal{U}\subseteq\Pi_{n}\subsetneq\mathcal{D}$}
    \end{array}
  \bigg\}
  \stackrel{1:1}{\longleftrightarrow}
  \bigg\{ \begin{array}{l}
    \mbox{finest $\infty$-admissible SODs of $\mathcal{D}/\mathcal{U}$}
     \end{array}
  \bigg\}
\]
{\em taking $(\Pi_{i};i\in\Phi_n) \mapsto (Q\langle\Pi_{i},\mathcal{U}\rangle;i\in\Phi_n\; \text{or}\;i\in\Phi_{n-1})$, which is compatible with mutations. }
\end{prop}

\begin{proof}
Let $(\Pi_{i};i\in\Phi_n)$ be a finest $\infty$-admissible SOD of $\mathcal{D}$ such that $\mathcal{U}\subseteq\Pi_{n}\subsetneq\mathcal{D}$.
If $\Pi_{n}\supsetneq\mathcal{U}$, then
by Proposition \ref{mutaion preserve finest}, each such SOD $(\Pi_{i};i\in\Phi_n)$
determines a unique finite finest $\infty$-s-admissible filtration $\mathscr{T}:\;
0=\mathcal{T}_0\subsetneq\mathcal{T}_1\subsetneq\cdots\subsetneq\mathcal{T}_{n-1}
\subsetneq\mathcal{T}_{n}=\mathcal{D}$,
where
$\mathcal{T}_{i}=\Pi_{\geq n-i+1}$ and $\mathcal{T}_{1}\supsetneq\mathcal{U}$.
By Proposition \ref{admiss filt bijection}, $\mathscr{T}$ corresponds uniquely to $Q\mathscr{T}$,
both of which are finite finest $\infty$-admissible. Applying Proposition \ref{mutaion preserve finest}
again, $Q\mathscr{T}$ gives rise to a unique finest $\infty$-admissible SOD
$(P_{i};i\in\Phi_n)$, where $P_{n}=Q\mathcal{T}_1=Q\Pi_{n}$ and
\begin{equation*}
P_{i}=(Q\mathcal{T}_{n-i})^{\bot}\cap Q\mathcal{T}_{n-i+1}=Q\langle\Pi^{\bot}_{\geq i+1}\cap\Pi_{\geq i},\mathcal{U}\rangle=
Q\langle\Pi_{i},\mathcal{U}\rangle\;\;
\text{for}\;\;1\leq i\leq n-1.
\end{equation*}

Otherwise, we have that $\Pi_{n}=\mathcal{U}$. Then, each such SOD $(\Pi_{i};i\in\Phi_n)$ corresponds uniquely to a finest $\infty$-admissible SOD $(Q\langle\Pi_{i},\mathcal{U}\rangle;i\in\Phi_{n-1})$ with $\Pi_{i}\not\supseteq\mathcal{U}$ for all $1\leq i\leq n-1$.
This establishes the desired bijection.
An analogous verification to Proposition \ref{admiss filt bijection} shows the compatibility of mutations.
\end{proof}

An $\infty$-admissible  subcategory $\mathcal{A}$ in $\mathcal{D}$  is called
\emph{finest} if there  is a finest $\infty$-admissible SOD $(\Pi_{i};i\in\Phi_n)$
such that $\Pi_{i}=\mathcal{A}$ for some $i$.
Let $\Sigma(\mathcal{D})$ denote the set of all finest $\infty$-admissible  SODs of $\mathcal{D}$, and $\mathscr{A}(\mathcal{D})$ the set of finest
$\infty$-admissible  subcategories in $\mathcal{D}$.
This gives the following corollary.

\begin{cor} \label{embedding and decomposition}
There is a decomposition of $\Sigma(\mathcal{D})$ given by
$$\Sigma(\mathcal{D})=
\mathop{\dot{\bigcup}}\limits_{\mathcal{U}\in\mathscr{A}(\mathcal{D})}^{}
Q_{\mathcal{U}}^{-1}(\Sigma(\mathcal{D}/ \mathcal{U})).$$
\end{cor}

\subsection{Mutation graphs}  \label{mutation graph}
Inspired by Sections \ref{Sect 4} and \ref{Sect 5}, we introduce the following definition.
\begin{defn}
The \emph{mutation graph}
of finest $\infty$-admissible SODs of a triangulated $\mathcal{D}$ is defined as follows:
\begin{itemize}
\item [(1)] Vertices are all
finest $\infty$-admissible  SODs of $\mathcal{D}$;
\item [(2)] Arrows correspond to right mutations between them.
\end{itemize}
\end{defn}
Similarly, we define the corresponding \emph{mutation graphs} for finite finest $\infty$-admissible t-stabilities and finite finest $\infty$-s-admissible filtrations, respectively.
If $\mathcal{D}$ admits a Serre functor, then  by  Lemma \ref{admissible sucat and sod gen by excep seq},  the vertices of each graph are exactly the finest ones.

As a consequence of Propositions  \ref{t-stab and sod} and \ref{mutaion preserve finest}, we have the following proposition.

\begin{prop}
Let $\mathcal{D}$  be a triangulated category.
Then there are one-to-one correspondences among the mutation graphs of
\begin{itemize}
  \item [(1)] finest $\infty$-admissible SODs of $\mathcal{D}$;
  \item [(2)] finite  finest $\infty$-s-admissible filtrations in $\mathcal{D}$;
  \item [(3)] equivalence classes of
 finite finest $\infty$-admissible t-stabilities on $\mathcal{D}$.
\end{itemize}
\end{prop}

The following result provides a characterization of the connectedness of the mutation graph of $\mathcal{D}$.

\begin{thm} \label{connectedness reduction}
Assume  the mutation graph is connected for every quotient $\mathcal{D}/ \mathcal{U}$ with $\mathcal{U}\in\mathscr{A}(\mathcal{D})$.
Then the mutation graph of $\mathcal{D}$ is connected if and only if
for any $\mathcal{U},\mathcal{V}\in\mathscr{A}(\mathcal{D})$, there exist $\mathcal{W}_i\in\mathscr{A}(\mathcal{D})$ for $1\leq i\leq m$ such that
$\mathcal{W}_{1}=\mathcal{U},\;\;
\mathcal{W}_{m}=\mathcal{V},\;\;
\text{and}\;\; \mathscr{A}(\mathcal{W}_i^{\bot})\cap
\mathscr{A}({}^{\bot}\mathcal{W}_{i\pm1})\neq\emptyset$.
\end{thm}

\begin{proof}
To prove the sufficiency, let $(\Pi_{i};i\in\Phi_n)$ and $(P_{j};j\in\Phi_s)$ be two finest $\infty$-admissible SODs of $\mathcal{D}$, with $\Pi_i=\mathcal{U}$ and $P_j=\mathcal{V}$ for some $i, j$.
By iteratively applying right/left mutations along $\mathcal{U}$ and $\mathcal{V}$,  we may assume $\Pi_{n}=\mathcal{U}$ and $P_{1}=\mathcal{V}$.

Then, there exist  $\mathcal{W}_1,\mathcal{W}_2\in\mathscr{A}(\mathcal{D})$ such that $\mathcal{K}_1\in\mathscr{A}(\mathcal{W}_1^{\bot})\cap
\mathscr{A}({}^{\bot}\mathcal{W}_{2})$ or $\mathcal{K}_1\in\mathscr{A}(\mathcal{W}_2^{\bot})\cap
\mathscr{A}({}^{\bot}\mathcal{W}_{1})$.
Without loss of generality, we treat the case $\mathcal{K}_1\in\mathscr{A}(\mathcal{W}_1^{\bot})\cap
\mathscr{A}({}^{\bot}\mathcal{W}_{2})$, the other is similar.
This gives an $\infty$-admissible SOD  $(Q_{k};k\in\Phi_t)$ with $Q_{k}=\mathcal{K}_1$ and $Q_{1}=\mathcal{W}_{2}$.
Repeated right mutations along $\mathcal{K}_1$ in $(Q_{k};k\in\Phi_s)$ gives $Q_{t}=\mathcal{K}_1$.
Analogously, since $\mathcal{W}_{1}=\mathcal{U}$ and $\mathcal{K}_1\in\mathscr{A}(\mathcal{W}_1^{\bot})$, there exists an $\infty$-admissible SOD $(\Pi'_{i};i\in\Phi_n)$ with $\Pi'_{n}=\mathcal{K}_1$ that is connected to $(\Pi_{i};i\in\Phi_n)$.

The connectedness of the subgraph with vertex set   $Q_{\mathcal{K}_1}^{-1}(\Sigma(\mathcal{D}/\mathcal{K}_1))$ implies that $(\Pi_{i};i\in\Phi_n)$ and $(Q_{k};k\in\Phi_t)$ are connected, which yields $n=t=s$.
Fix an integer $n$, and let $(\Pi^{(j)}_{i};i\in\Phi_n)$ for $1\leq j\leq m$ be the SODs corresponding to $\mathcal{W}_j\in\mathscr{A}(\mathcal{D})$, where $\Pi^{(1)}_{i}=\Pi_{i}$ and $\Pi^{(m)}_{i}=P_{i}$ for $i\in\Phi_n$, and let $(Q^{(j)}_{i};i\in\Phi_n)$ for $1\leq j\leq m-1$ be the SODs corresponding to $\mathcal{K}_{j}\in\mathscr{A}(\mathcal{W}_j^{\bot})\cap
\mathscr{A}({}^{\bot}\mathcal{W}_{j\pm1})$.
By iterating the above argument, each $(\Pi^{(j)}_{i};i\in\Phi_n)$ is connected to $(\Pi^{(j+1)}_{i};i\in\Phi_n)$, leading to a path between $(\Pi_{i};i\in\Phi_n)$ and $(P_{i};i\in\Phi_n)$ such that the subgraph of these vertices has the shape depicted below,
\begin{equation*}
\xymatrix@C=2.25em@R=5.5ex@M=0.1pt@!0{
\cdots&&&&(Q^{(1)}_{i};i\in\Phi_n)\ar@{->}[rrdd]^{\rho^{\pm}}&&
&&&\cdots&\cdots&\ar@{->}[rrdd]^-{\rho^{\pm}}&&&&
(Q^{(m-1)}_{i};i\in\Phi_n)\ar@{->}@<1.7pt>[rrrdd]^-{\rho^{\pm}}
&&&&&
\cdots\\
&&&&&&&&&&&&&&&&&&&&&&&&  \\
\cdots&& (\Pi^{}_{i};i\in\Phi_n)\ar@{->}[rruu]^-{\rho^{\pm}}&&&&
(\Pi^{(2)}_{i};i\in\Phi_n)\ar@{->}[rruu]^-{\rho^{\pm}}&&&\cdots&\cdots
&&& (\Pi^{(m-1)}_{i};i\in\Phi_n)\ar@{->}[rruu]^-{\rho^{\pm}}&&&&&
(P_{i};i\in\Phi_n)&&\dots
&&  \\ }
\end{equation*}
where $\rho^{\pm}$ denotes a finite sequence
of left/right mutations.
Hence, the mutation graph of $\mathcal{D}$ is connected.

To prove the necessity, let $\mathcal{U},\mathcal{V}\in\mathscr{A}(\mathcal{D})$.
Then there exist SODs $(\Pi^{(1)}_{i};i\in\Phi_n)$ and $(\Pi^{(m)}_{j};j\in\Phi_n)$ with $\mathcal{U}=\Pi^{(1)}_{n}$ and $\mathcal{V}=\Pi^{(m)}_{1}$, and a mutation path passing through all  $(\Pi^{(k)}_{i};i\in\Phi_n)$ for $1\leq k\leq m$.
By direct calculation, the right/left mutation from $(\Pi^{(k)}_{i};i\in\Phi_n)$ to
$(\Pi^{(k+1)}_{i};i\in\Phi_n)$ yields $\mathscr{A}(\mathcal{W}_k^{\bot})\cap
\mathscr{A}({}^{\bot}\mathcal{W}_{k+1})\neq\emptyset$, where $\mathcal{W}_k=\Pi^{(k)}_{n}$ and $\mathcal{W}_{k+1}=\Pi^{(k+1)}_{1}$, or $\mathcal{W}_k=\Pi^{(k+1)}_{n}$ and $\mathcal{W}_{k+1}=\Pi^{(k)}_{1}$.
\end{proof}

\begin{rem}  \label{mutation graph of subgraphs}
Under the assumption, the relation $\mathscr{A}(\mathcal{W}_i^{\bot})\cap
\mathscr{A}({}^{\bot}\mathcal{W}_{i\pm1})\neq\emptyset$ can be interpreted geometrically.
It means that the subgraphs consisting of $Q_{\mathcal{W}_i}^{-1}(\Sigma(\mathcal{D}/ \mathcal{W}_i))$ and $Q_{\mathcal{W}_{i\pm1}}^{-1}(\Sigma(\mathcal{D}/ \mathcal{W}_{i\pm1}))$, respectively, are connected by a finite sequence of connected subgraphs.

Thus, the mutation graph of $\mathcal{D}$  is connected precisely when so is the \emph{component graph}.
Namely, the vertices of this graph are the subgraphs $Q_{\mathcal{U}}^{-1}(\Sigma(\mathcal{D}/\mathcal{U}))$ for $\mathcal{U} \in \mathscr{A}(\mathcal{D})$, and there is an arrow $\rho_j$ from $Q_{\mathcal{U}}^{-1}(\Sigma(\mathcal{D}/\mathcal{U}))$ to $Q_{\mathcal{V}}^{-1}(\Sigma(\mathcal{D}/\mathcal{V}))$ if and only if  there exist SODs $(\Pi_{i};i\in\Phi_n)$ and $(P_{i};i\in\Phi_n)$ in the respective subgraphs such that $\rho_j(\Pi^{}_{i};i\in\Phi_n)=(P_{i};i\in\Phi_n)$, which implies  $\Pi_{1}=P_{1}$.
\end{rem}

\section{Examples}

In this section we present examples with detailed constructions.

\subsection{$A_2$ case}
Let $Q: 1\rightarrow2$ and $A_2$:=$\mathbf{k}Q$.
Then the Auslander-Reiten (AR) quiver $\Gamma(\mathcal{D})$ of the derived category $\mathcal{D}=\mathcal{D}^{b}({\rm{mod}}$-$A_2)$ has the form
$$\xymatrix@M=1pt@!0{
&&&S_{2}[-1]\ar[dr]&&S_{1}[-1]\ar[dr]    &  & P_{1} \ar[dr]_{}  &  &S_{2}[1]\ar[dr]&&S_{1}[1]\ar[dr]_{}  &  & \\
&\cdots&\ar[ur]^{}&  &P_{1}[-1]\ar[ur]^{}  &   & S_{2}\ar[ur]^{} &  & S_{1}\ar[ur]^{}  && P_{1}[1] \ar[ur]^{}&& &\cdots& \\ }
$$
There are exactly three indecomposable modules in ${\rm{mod}}$-$A_2$, namely, the simple modules $S_1, S_2$ and the projective modules $P_1, P_2(=S_2)$. Moreover, there are only three equivalence classes of finite finest t-stability  $(\Phi_2,\{\Pi_{i}\}_{i\in\Phi_2})$ on $\mathcal{D}^{b}({\rm{mod}}$-$A_2)$ given as follows:
\begin{itemize}
  \item [(1)]  $\Pi_{ 1}=\langle S_1\rangle,
  \Pi_{ 2}=\langle S_2\rangle$; Up to shifts,  the  indecomposable module $P_{1}$ is the only one that  does not belong to $\Pi_{i}$ for $i=1, 2$, whose HN-filtration is given by
$S_2 \rightarrow P_1 \rightarrow S_1 \rightarrow S_2[1]$;
  \item [(2)] $\Pi_{ 1}=\langle S_2\rangle,
  \Pi_{2}=\langle P_{1}\rangle$; Up to shifts, the indecomposable module $S_1$ is the only one not in $\Pi_{i}$ for $i=1, 2$
  with an HN-filtration given by
$P_1 \rightarrow S_1 \rightarrow S_2[1] \rightarrow P_1[1]$;
  \item [(3)] $\Pi_{ 1}=\langle P_{1}\rangle,
  \Pi_{2}=\langle S_1\rangle$; Up to shifts,  the  indecomposable module $S_2$ is the only one not in $\Pi_{i}$ for $i=1, 2$ that has an HN-filtration given by
$S_1[-1] \rightarrow S_2 \rightarrow P_{1} \rightarrow S_1$.
\end{itemize}

As a consequence of  Propositions \ref{t-stab and sod} and \ref{any sod construction}, we list all (finest) SODs of $\mathcal{D}^{b}({\rm{mod}}$-$A_2)$:
$$(a)\; (\langle S_1 \rangle,\langle S_2\rangle), \quad\quad\;
(b)\; (\langle S_2 \rangle,\langle P_1\rangle),\quad\quad\;
(c)\; (\langle P_1 \rangle,\langle S_1\rangle).$$

The mutation graphs of finest SODs, as well as those of finite finest t-stabilities and finite finest admissible filtrations, are given below.
Here, we write $\Pi_1<\Pi_2$ to denote a finite finest t-stability $(\Phi_2,\{\Pi_{i}\}_{i\in\,\Phi_2})$.

\begin{center}
\scalebox{0.95}{
$
\xymatrix@C0.000001pt@R0.5pt{
&&&&&&&&&&&{\begin{smallmatrix}(\langle S^{}_{1}\rangle,\langle S^{}_{2}\rangle)
\end{smallmatrix}}\ar@/^/[rrdddd]&&&&&&&\\
&&&&&&&&&&\\
&&&&&&&&&&\\
&&&&&&&&&&\\
&&&&&&&&&{\begin{smallmatrix}(\langle S^{}_{2}\rangle,\langle P^{}_{1}\rangle)
\end{smallmatrix}}\ar@/^/[rruuuu]&&&&
{\begin{smallmatrix}(\langle P^{}_{1}\rangle,\langle S^{}_{1}\rangle)
\end{smallmatrix}}\ar@/^/[llll]
&&&&&&&&&&\\
}$
}
\end{center}
\vspace{0.15cm}
\[
\]

\begin{center}
\scalebox{0.95}{
$
\xymatrix@C0.000001pt@R0.5pt{
&&&&&&&&&&\\
&&&&&&&&{\begin{smallmatrix}\langle S^{}_{1}\rangle\;\,<\;\,\langle S^{}_{2}\rangle
\end{smallmatrix}}\ar@/^/[rrdddd]&&&&&&&&&&&
{\begin{smallmatrix}
0\;\,\subsetneq\;\,\langle S^{}_{2}\rangle\;\,\subsetneq\;\,\mathcal{D}
\end{smallmatrix}}\ar@/^/[rrdddd]&&&\\
&&&&&&&&&&&&&&\\
&&&&&&&&&&&&&&\\
&&&&&&&&&&&&&&\\
&&&&&&{\begin{smallmatrix}\langle S^{}_{2}\rangle\;\,<\;\,\langle P^{}_{1}\rangle
\end{smallmatrix}}\ar@/^/[rruuuu]&&&&
{\begin{smallmatrix}\langle P^{}_{1}\rangle\;\,<\;\,\langle S^{}_{1}\rangle
\end{smallmatrix}}\ar@/^/[llll]
&&&&&&&
{\begin{smallmatrix}0\;\,\subsetneq\;\,\langle P^{}_{1}\rangle\;\,\subsetneq\;\,\mathcal{D}
\end{smallmatrix}}\ar@/^/[rruuuu]&&&&
{\begin{smallmatrix}0\;\,\subsetneq\;\,\langle S^{}_{1}\rangle\;\,\subsetneq\;\,\mathcal{D}
\end{smallmatrix}}\ar@/^/[llll]&&&\\
}$
}
\end{center}
\vspace{0.15cm}

\subsection{$A_3$ case}

Let $Q: 1\rightarrow2\rightarrow3$ and $A_3$:=$\mathbf{k}Q$.
Then the Auslander-Reiten quiver $\Gamma(\mathcal{D})$ of the derived category  $\mathcal{D}=\mathcal{D}^{b}({\mathrm{mod}}$-$A_3)$ has the form
$$
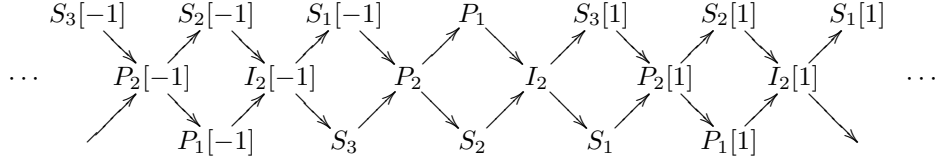
\xymatrix@M=1pt@!0{
&S_{3}[-1]\ar[dr]&&S_{2}[-1]\ar[dr]&&S_{1}[-1]\ar[dr]    &  & P_{1} \ar[dr]_{}  &  &S_{3}[1]\ar[dr]&&S_{2}[1]\ar[dr]_{}  &  &S_{1}[1]  \\
\cdots&&P_{2}[-1]\ar[dr]_{} \ar[ur]^{} &  & I_{2}[-1]\ar[dr]_{}\ar[ur]^{}  &   & P_{2}\ar[dr]_{} \ar[ur]^{} &  & I_{2}\ar[dr]_{}\ar[ur]^{}  && P_{2}[1]\ar[dr]_{} \ar[ur]^{}&& I_{2}[1]\ar[dr]\ar[ur]^{}&&\cdots \\
&\ar[ur]&&P_{1}[-1] \ar[ur]^{} &  &   S_{3} \ar[ur]^{} &  & S_{2} \ar[ur]^{} &  & S_{1}  \ar[ur]^{} &  & P_{1}[1] \ar[ur]^{} &  & }$$

Let $\Phi_3 =\{ 1, 2, 3\}$ be the linear ordered set, and let $\Pi_{1}=\langle P_1\rangle, \Pi_{2}=\langle S_2\rangle, \Pi_{3}=\langle I_2\rangle$, then $(\Phi_3,\{\Pi_{i}\}_{i\in\,\Phi_3})$ is a finite finest t-stability on $\mathcal{D}$.
Up to shifts,  the  indecomposable modules $S_3, P_{2}, S_{1}$ are the only ones that  do not belong to $\Pi_{i}$ for $i=1, 2, 3$, whose HN-filtration are given respectively by:
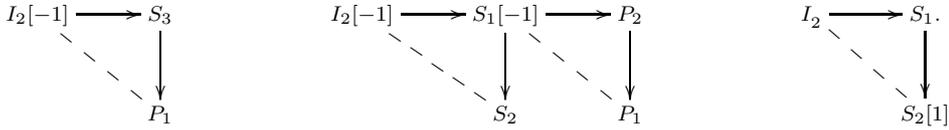
\begin{equation*}
{\footnotesize\xymatrix { I_{2}[-1]\ar@{->}[r]^-{}\ar@{--}[rd] &S_{3}\ar@{->}[d]^{}&&
I_{2}[-1]\ar@{->}[r]^-{}\ar@{--}[rd]&S_{1}[-1]\ar@{->}[d]^{}\ar@{->}[r]^-{}\ar@{--}[rd] &P_{2}\ar@{->}[d]^{}&&
I_{2}^{}\ar@{->}[r]^-{}\ar@{--}[rd]&S_{1}.\ar@{->}[d]^{}&&\\
&P_{1}&&&S_{2}&P_{1}&&&S_{2}[1] & }}
\end{equation*}

Similarly,  we can obtain all the other equivalence classes of finite finest t-stabilities.
As a result, there are sixteen classes of  finest SODs $(\Pi_{1},\Pi_{2},\Pi_{3})$ of $\mathcal{D}^{b}({\mathrm{mod}}$-$A_3)$ as in Table \ref{A3}, where each line determines an equivalence class.

\begin{table}[h]
\caption{Sixteen finest SODs $(\Pi_{1},\Pi_{2},\Pi_{3})$ of $\mathcal{D}^{b}({\rm{mod}}$-$A_3)$}
\begin{tabular}{|c|c|c|c|c|c|c|c|}
\hline
\makecell*[c]{$\Pi_{1}$}&$\Pi_{2}$&$\Pi_{3}$
&\multirow{8}{*}{} &
$\Pi_{1}$&$\Pi_{2}$&$\Pi_{3}$\\
\cline{1-3}  \cline{5-7}
\makecell*[c]{$\langle P_1\rangle$}&$ \langle I_2\rangle $&$\langle S_1\rangle$&&
$\langle S_1\rangle$&$ \langle P_2\rangle $&$\langle S_2\rangle$\\
\cline{1-3}  \cline{5-7}
\makecell*[c]{$\langle I_2\rangle$}&$ \langle S_3\rangle $&$\langle S_1\rangle$&&
$\langle P_1\rangle$&$ \langle S_1\rangle $&$\langle S_2\rangle$\\
\cline{1-3}  \cline{5-7}
\makecell*[c]{$\langle S_3\rangle$}&$ \langle P_1\rangle $&$\langle S_1\rangle$&&
$\langle P_2\rangle$&$ \langle P_1\rangle $&$\langle S_2\rangle$\\
\cline{1-3}  \cline{5-7}
\makecell*[c]{$\langle S_3\rangle$}&$ \langle P_2\rangle $&$\langle P_1\rangle$&&
$\langle S_1\rangle$&$ \langle S_2\rangle $&$\langle S_3\rangle$\\
\cline{1-3}  \cline{5-7}
\makecell*[c]{$\langle P_2\rangle$}&$ \langle S_2\rangle $&$\langle P_1\rangle$&&
$\langle S_2\rangle$&$ \langle I_2\rangle $&$\langle S_3\rangle$\\
\cline{1-3}  \cline{5-7}
\makecell*[c]{$\langle S_2\rangle$}&$ \langle S_3\rangle $&$\langle P_1\rangle$&&
$\langle I_2\rangle$&$ \langle S_1\rangle $&$\langle S_3\rangle$\\
\cline{1-3}  \cline{5-7}
\makecell*[c]{$\langle S_1\rangle$}&$ \langle S_3\rangle $&$\langle P_2\rangle$&&
$\langle P_1\rangle$&$ \langle S_2\rangle $&$\langle I_2\rangle$\\
\cline{1-3}  \cline{5-7}
\makecell*[c]{$\langle S_3\rangle$}&$ \langle S_1\rangle $&$\langle P_2\rangle$&&
$\langle S_2\rangle$&$ \langle P_1\rangle $&$\langle I_2\rangle$\\
\hline
\end{tabular}
\label{A3}
\end{table}

Moreover, for any finest SOD $(\Pi_{1},\Pi_{2},\Pi_{3})$, each $\Pi_i$ is generated by a single object $E_i$.
By identifying each finest SOD $(\langle E_1\rangle, \langle E_2\rangle, \langle E_3\rangle)$ with $(E_1,E_2,E_3)$,
the mutation graph of finest SODs of $\mathcal{D}^{b}({\rm{mod}}$-$A_3)$ is depicted below.

\begin{center}
\[\xymatrix@C0.00001pt@R0.9pt{
&&&&&&&{\begin{smallmatrix}(P_{1} ,I_2, S_1)\ar@[blue][lllllllddddddddddd]\ar[rrrrrrrrr]
\end{smallmatrix}}&&&&&&&&&
{\begin{smallmatrix}(S_{3} ,P_1, S_1)
\end{smallmatrix}}\ar@[blue][rrrrrddddddddddd]\ar[llldd]&&&&&\\
&&&&&&&&&&&&&&&&&&&&&&&&&&&\\
&&&&&&&&&&&&&
{\begin{smallmatrix}(I_{2} ,S_3, S_1)
\end{smallmatrix}}\ar[lllllluu]\ar@[blue]@/^/[rrdd]&&&&&&&\\
&&&&&&&&&&&&&&&&&&&&&&&\\
&&&&&&&&&&
{\begin{smallmatrix}(P_{2} ,P_1, S_2)
\end{smallmatrix}}\ar[llldddddddddd]
\ar@[blue]@/^/[rrrrrrdddddddddd]&&&&&
{\begin{smallmatrix}(I_{2} ,S_1, S_3)
\end{smallmatrix}}\ar@[blue]@/^/[lluu]\ar[lldddddddddddddddd]&&&\\
&&&&&&&&&&&&&&&&&&&&&&\\
&&&&&&&&&&&&&&&&&&&&&&&\\
&&&&&&&{\begin{smallmatrix}(P_{1} ,S_1, S_2)
\end{smallmatrix}}\ar@[blue][uuuuuuu]\ar[rrruuu]&&&&&&&&&
{\begin{smallmatrix}(S_{3} ,S_1, P_2)
\end{smallmatrix}}\ar@[blue][uuuuuuu]\ar@/^/[lllllldd]&&&\\
&&&&&&&&&&&&&&&\\
&&&&&&&&&&
{\begin{smallmatrix}(S_{1} ,S_3, P_2)
\end{smallmatrix}}\ar@/^/[rrrrrruu]\ar@[blue][dddddddd]&&\\
&&&&&&&&&&&&&&&\\
{\begin{smallmatrix}(P_{1} ,S_2, I_2)
\end{smallmatrix}}\ar@[blue][rrrrrrruuuu]\ar@/^/[rrrrrrrddddddddddd]
&&&&&&&&&&&&&
&&&&&&&&
{\begin{smallmatrix}(S_{3} ,P_2, P_1)
\end{smallmatrix}}\ar@[blue][llllluuuu]\ar[lllllddddddddddd]&&\\
&&&&&&&&&&&&&\\
&&&&&&&&&&&&&&\\
&&&&&&&{\begin{smallmatrix}(S_{1} ,P_2, S_2)
\end{smallmatrix}}\ar[uuuuuuu]\ar@[blue][rrruuuuu]&&&&&&&&&
{\begin{smallmatrix}(P_{2} ,S_2, P_1)
\end{smallmatrix}}\ar@[blue]@/^/[lllllluuuuuuuuuu]\ar[rrrrruuu]&&\\
&&&&&&&&&&&&&&&&&&&&\\
&&&&&&&&&&&&&&&&&&&&\\
&&&&&&&&&&{\begin{smallmatrix}(S_{1} ,S_2, S_3)
\end{smallmatrix}}\ar@[blue][llluuu]\ar[rrrrruuuuuuuuuuuuu]&\\
&&&&&&&&&&&&&&&\\
&&&&&&&&&&&&&&\\
&&&&&&&&&&&&&
{\begin{smallmatrix}(S_{2} ,I_2, S_3)
\end{smallmatrix}}\ar[llluuu]\ar@[blue][lllllldd]&&&&&\\
&&&&&&&&&&&&&&&&&&\\
&&&&&&&{\begin{smallmatrix}(S_{2} ,P_1, I_2)
\end{smallmatrix}}\ar@[blue][rrrrrrrrr]\ar@/^/[llllllluuuuuuuuuuu]&&&&&&&&&
{\begin{smallmatrix}(S_{2} ,S_3, P_1)
\end{smallmatrix}}\ar@[blue][llluu]\ar[uuuuuuuu]&&\\
  }\]
\end{center}

Recall that $\Sigma(\mathcal{D})$ denotes the set of all finest SODs of $\mathcal{D}$.
By Proposition \ref{reduction thm}, each $\Sigma(\mathcal{D}/\langle E_{i}\rangle)$ embeds into $\Sigma(\mathcal{D})$, leading to the following decomposition
\begin{equation*}
\scalebox{0.84}{$
Q^{-1}_{\langle S_{1}\rangle}(\Sigma(\mathcal{D}/\langle S_{1}\rangle))\;\;\dot{\bigcup}\;\;
Q^{-1}_{\langle S_{2}\rangle}(\Sigma(\mathcal{D}/\langle S_{2}\rangle))\;\;\dot{\bigcup}\;\; Q^{-1}_{\langle S_{3}\rangle}(\Sigma(\mathcal{D}/\langle S_{3}\rangle)) \dot{\bigcup}\;\;Q^{-1}_{\langle P_{1}\rangle}(\Sigma(\mathcal{D}/\langle P_{1}\rangle))\;\;\dot{\bigcup}\;\; Q^{-1}_{\langle P_{2}\rangle}(\Sigma(\mathcal{D}/\langle P_{2}\rangle))\;\;\dot{\bigcup}\;\; Q^{-1}_{\langle I_{2}\rangle}(\Sigma(\mathcal{D}/\langle I_{2}\rangle)),
$}
\end{equation*}
where each quotient $\mathcal{D}/\langle E_{i}\rangle$ is of type $A_2$ or $A_1\times A_1$.

Since the braid group action on full exceptional sequences in ${\mathrm{mod}}$-$A_3$ is transitive \cite{crbo}, it follows from Theorem \ref{finest sod and full exceptional sequence} that the mutation graph of $\Sigma(\mathcal{D})$ is connected.
Clearly, each subgraph with vertex set $Q_{\langle E_{i}\rangle}^{-1}(\Sigma(\mathcal{D}/\langle E_{i}\rangle))$ is connected, so by Theorem \ref{connectedness reduction}, the component graph in Remark \ref{mutation graph of subgraphs} (with blue arrows) is as well.

\subsection{Weighted projective line of weight type $(2)$ case}

Let $\mathbb{X}=\mathbb{X}(2)$ be a weighted projective line of weight type $(2)$.
Recall from \cite{gl} that the group  $\mathbb{L} = \mathbb{L}(2)$ is the rank one abelian group with generators
$\vec{x}_{1}, \vec{x}_{2}$ and the relations $$ 2\vec{x}_{1}=\vec{x}_{2}:=\vec{c}.$$

Each element $\vec{x}\in\mathbb{L}$ has the \emph{normal form}
$\vec{x}=l\vec{c}$ or $\vec{x}=\vec{x}_1+l\vec{c}$ for some $l\in\mathbb{Z}$. In this case, each indecomposable bundle is a line
bundle, hence has the form $\mathcal{O}(\vec{x}), \vec{x} \in \mathbb{L}$.
The AR-quiver $\Gamma(\coh \mathbb{X}(2))$ of $\coh \mathbb{X}(2)$ has the following shape, cf. \cite{lenz}:
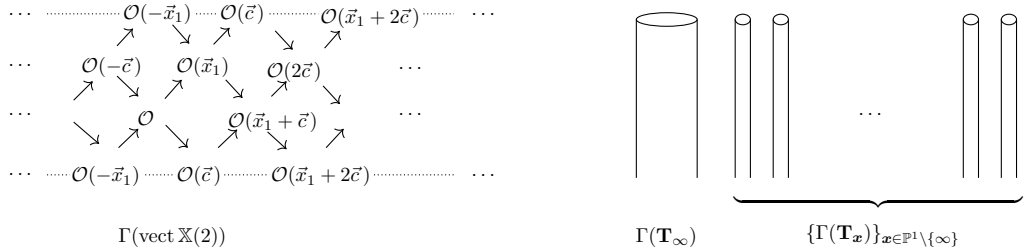
$$\begin{tikzpicture}[row sep=17cm, column sep=1.2cm]
\tikzstyle{every node}=[font=\large,scale=0.69]
\node (O) at (-10.48,-0.8) [above] {$\cdots$};

\draw[->] (-9.8,-0.47) --(-9.55,-0.22);

\draw[->] (-9.8,-0.76) --(-9.465,-1.05);

\node (O) at (-10.48,0.52) [above] {$\cdots$};

\draw[densely dotted] (-10.15,0.675) --(-9.167,0.675);

\draw[densely dotted] (-10.15,-1.46) --(-9.8519,-1.46);

\node (O) at (-10.48,-1.6) [above] {$\cdots$};

\node (O) at (-10.48,-0.15) [above] {$\cdots$};

\node (O) at (-9.38,-1.7) [above] {$\mathcal{O}(-\vec{x}_{1})$};

\draw[->] (-9.21,-1.1) --(-8.95,-0.85);

\node (O) at (-9.3,-0.25) [above] {$\mathcal{O}(-\vec{c}\,)$};

\draw[->] (-9.23,0.25) --(-9.02,0.48);

\draw[->] (-9.23,-0.15) --(-8.95,-0.42);

\node (O) at (-8.85,-0.89) [above] {$\mathcal{O}$};

\draw[->] (-8.6,-0.45) --(-8.38,-0.2);

\draw[->] (-8.59,-0.8) --(-8.3,-1.1);

\draw[densely dotted] (-7.82,-1.46) --(-7.2285,-1.46);

\node (O) at (-8.15,-1.7) [above] {$\mathcal{O}(\vec{c}\,)$};

\draw[densely dotted] (-8.859,-1.46) --(-8.465,-1.46);

\draw[->] (-7.9,-1.1) --(-7.65,-0.86);

\draw[->] (-7.8,-0.2) --(-7.58,-0.44);

\node (O) at (-8.1,-0.25) [above] {$\mathcal{O}(\vec{x}_{1})$};

\draw[->] (-7.96,0.27) --(-7.75,0.49);

\draw[densely dotted] (-8.19,0.675) --(-7.91,0.675);

\node (O) at (-7.58,0.45) [above] {$\mathcal{O}(\vec{c}\,)$};

\draw[->] (-8.5,0.53) --(-8.28,0.28);

\node (O) at (-8.7,0.45) [above] {$\mathcal{O}(-\vec{x}_{1})$};

\draw[densely dotted] (-7.225,0.675) --(-6.583,0.675);

\draw[->] (-7.15,0.45) --(-6.92,0.2);

\node (O) at (-7.18,-0.98) [above] {$\mathcal{O}(\vec{x}_{1}+\vec{c}\,)$};

\draw[->] (-7.2,-0.42) --(-7.0,-0.22);

\draw[->] (-7.25,-0.89) --(-7.0,-1.12);

\node (O) at (-6.55,-1.7) [above] {${\tiny\mathcal{O}(\vec{x}_{1}+2\vec{c}\,)}$};

\draw[densely dotted] (-5.85,-1.46) --(-4.8,-1.46);

\draw[->] (-6.52,-0.28) --(-6.3,-0.52);

\draw[->] (-6.48,-1.1) --(-6.25,-0.85);

\node (O) at (-6.9,-0.32) [above] {$\mathcal{O}(2\vec{c}\,)$};

\draw[->] (-6.5,0.22) --(-6.3,0.45);

\draw[densely dotted] (-5.2,0.675) --(-4.8,0.675);

\node (O) at (-5.9,0.4) [above] {$\mathcal{O}(\vec{x}_1+2\vec{c}\,)$};

\node (O) at (-4.39,0.52) [above] {$\cdots$};

\node (O) at (-4.39,-1.61) [above] {$\cdots$};

\node (O) at (-8.5,-2.5) [above] {$\Gamma(\vect \mathbb{X}(2))$};

\node (O) at (-5.35,-0.2) [above] {$\cdots$};

\node (O) at (-5.35,-0.8) [above] {$\cdots$};

\draw (-2,0.6) ellipse (0.4 and 0.09 );
\draw (-2.4,0.6) --(-2.4,-1.5);
\draw (-1.6,0.6) --(-1.6,-1.5);

\node (O) at (-2.006,-2.5) [above] {$\Gamma({{\bf{T}}}_{\infty})$};

\draw (-1,0.6) ellipse (0.1 and 0.07 );
\draw (-1.1,0.6) --(-1.1,-1.5);
\draw (-0.9,0.6) --(-0.9,-1.5);

\draw (-0.5,0.6) ellipse (0.1 and 0.07 );
\draw (-0.6,0.6) --(-0.6,-1.5);
\draw (-0.4,0.6)  --(-0.4,-1.5);

\node (O) at (0.7,-0.8) [above] {$\cdots$};

\draw (2.0,0.6) ellipse (0.1 and 0.07 );
\draw (2.1,0.6) --(2.1,-1.5);
\draw (1.9,0.6) --(1.9,-1.5);

\draw (2.5,0.6) ellipse (0.1 and 0.07 );
\draw (2.6,0.6) --(2.6,-1.5);
\draw (2.4,0.6) --(2.4,-1.5);

\node[rotate = 0] at (0.78,-1.8) {$\underbrace{\hspace{5.5cm}}$};

\node (O) at (0.84,-2.5) [above] {$\{\Gamma({{\bf{T}}_{{\bm{x}}}) \}_{{\bm{x}} \in \mathbb{P}^{1}\setminus\{\infty\} }}$};

\end{tikzpicture}$$
where  the horizontal dotted lines are identified in $\Gamma(\vect \mathbb{X}(2))$,
and $\Gamma({\rm{coh}}_0 \mathbb{X}(2))=\Gamma({{\bf{T}}}_{\infty})\cup
\{\Gamma({{\bf{T}}_{{\bm{x}}}) \}_{{\bm{x}} \in \mathbb{P}^{1}\setminus\{\infty\} }}$.
Here, by ${\mathrm{coh}}_0 \mathbb{X}(2)$ and $\vect \mathbb{X}(2)$ we mean the full subcategories of torsion sheaves and vector bundles, respectively. In particular, ${{\bf{T}}}_{\infty}$ is a tube category of rank two generated by the simples $S_{1,0}, S_{1,1}$, and ${{\bf{T}}}_{{\bm{x}}}$ is a tube category of rank one generated by the simple $S_{{\textbf{\textit{x}}}}$ for any ${\bm{x}} \in \mathbb{P}^{1}\setminus\{\infty\}$.

By the identification of each finest SOD $(\langle E_1\rangle, \langle E_2\rangle,\langle E_3\rangle)$ with its corresponding full exceptional sequence $(E_1, E_2,E_3)$, the mutation graph of the finest SODs of $\mathcal{D}=\mathcal{D}^{b}(\coh \mathbb{X}(2))$ has the shape depicted below.

\begin{center}
\[\xymatrix@C0.00001pt@R0.81pt{
&&&&&&&&&&&&&&&&&&&&&&&\\
&&&&&&\vdots&&&&\vdots&&&&\vdots&&\\
&&&&&&&&&&&&&&&\\
&&&&&&&&&&&&&&&\\
&&&&&&&&&&&&&&&\\
&&&&&&&&&&&&&&&\\
&&&&&&&&&&&&&&&\\
&&&&&&&&&&&&&&&\\
&&&&&&&&&&&&&&&\\
&&&&&&&&&&&&&&&\\
&&&&&&
{\begin{smallmatrix}
(\mathcal{O}(-\vec{c}),\mathcal{O},
S_{1,0})
\end{smallmatrix}}\ar@[blue][rrrrdddd]
&&&&
{\begin{smallmatrix}
(\mathcal{O}(-\vec{c}),S_{1,0}, \mathcal{O}(\vec{x}_1-\vec{c}))
\end{smallmatrix}}\ar[uuuu]\ar@[blue][llll]
&&\\
&&&&&&&&&&&&&\\
&&&&&&&&&&&&&\\
&&&&&&&&&&&&&\\
&&&&&&&&&&
{\begin{smallmatrix}
(\mathcal{O}(-\vec{c}),
\mathcal{O}(\vec{x}_1-\vec{c}),
\mathcal{O})
\end{smallmatrix}}\ar@[blue][uuuu]\ar[rrrr]
&&&&
{\begin{smallmatrix}
(S_{1,1},\mathcal{O}(-\vec{c}),
\mathcal{O})
\end{smallmatrix}}\ar[lllldddd]
&&\\
&&&&&&&&&&&&\\
&&&&&&&&&&&&\\
&&&&&&&&&&&&\\
&&&&&&
{\begin{smallmatrix}
(\mathcal{O}(\vec{x}_1-\vec{c}),
\mathcal{O}(\vec{x}_1),
S_{1,1})
\end{smallmatrix}}\ar@[blue][rrrrdddd]
\ar@/^4pc/[rrrruuuuuuuuuuuuuuuullll]
&&&&
{\begin{smallmatrix}
(\mathcal{O}(\vec{x}_1-\vec{c}),S_{1,1},
\mathcal{O})
\end{smallmatrix}}\ar[uuuu]\ar@[blue][llll]
&&\\
&&&&&&&&&&\\
&&&&&&&&&&\\
&&&&&&&&&&\\
&&&&&&&&&&
{\begin{smallmatrix}
(\mathcal{O}(\vec{x}_1-\vec{c}),\mathcal{O},
\mathcal{O}(\vec{x}_1))
\end{smallmatrix}}\ar@[blue][uuuu]\ar[rrrr]
&&&&
{\begin{smallmatrix}
(S_{1,0},
\mathcal{O}(\vec{x}_1-\vec{c}),
\mathcal{O}(\vec{x}_1))
\end{smallmatrix}}\ar[lllldddd]
\ar@[blue]@/_4pc/[rrrruuuuuuuuuuuuuuuullll]
&&\\
&&&&&&&&&&\\
&&&&&&&&&&\\
&&&&&&&&&&\\
&&&&&&
{\begin{smallmatrix}
(\mathcal{O},\mathcal{O}(\vec{c}),
S_{1,0})
\end{smallmatrix}}\ar@[blue][rrrrdddd]
\ar@/^4pc/[rrrruuuuuuuuuuuuuuuullll]
&&&&
{\begin{smallmatrix}
(\mathcal{O},S_{1,0},
\mathcal{O}(\vec{x}_1))
\end{smallmatrix}}\ar[uuuu]\ar@[blue][llll]
&&\\
&&&&&&&&&&&&\\
&&&&&&&&&&&&\\
&&&&&&&&&&&&\\
&&&&&&&&&&
{\begin{smallmatrix}
(\mathcal{O},\mathcal{O}(\vec{x}_1),
\mathcal{O}(\vec{c}))
\end{smallmatrix}}\ar@[blue][uuuu]\ar[rrrr]
&&&&
{\begin{smallmatrix}
(S_{1,1},\mathcal{O},
\mathcal{O}(\vec{c}))
\end{smallmatrix}}\ar[lllldddd]
\ar@[blue]@/_4pc/[rrrruuuuuuuuuuuuuuuullll]
&&\\
&&&&&&&&&&&&&\\
&&&&&&&&&&&&&\\
&&&&&&&&&&&&&\\
&&&&&&
{\begin{smallmatrix}
(\mathcal{O}(\vec{x}_1),
\mathcal{O}(\vec{x}_1+\vec{c}),
 S_{1,1})
\end{smallmatrix}}\ar@[blue][rrrrdddd]
\ar@/^4pc/[rrrruuuuuuuuuuuuuuuullll]
&&&&
{\begin{smallmatrix}
(\mathcal{O}(\vec{x}_1),S_{1,1},
\mathcal{O}(\vec{c}))
\end{smallmatrix}}\ar[uuuu]\ar@[blue][llll]
&&\\
&&&&&&&&&&&\\
&&&&&&&&&&&\\
&&&&&&&&&&&\\
&&&&&&&&&&
{\begin{smallmatrix}
(\mathcal{O}(\vec{x}_1),\mathcal{O}(\vec{c}),
\mathcal{O}(\vec{x}_1+\vec{c}))
\end{smallmatrix}}\ar@[blue][uuuu]\ar[rrrr]
&&&&
{\begin{smallmatrix}
(S_{1,0},\mathcal{O}(\vec{x}_1),
\mathcal{O}(\vec{x}_1+\vec{c}))
\end{smallmatrix}}\ar[lllldddd]\quad
\ar@[blue]@/_4pc/[rrrruuuuuuuuuuuuuuuullll]
&&\\
&&&&&&&&&\\
&&&&&&&&&\\
&&&&&&&&&\\
&&&&&&&&&&
{\begin{smallmatrix}
(\mathcal{O}(\vec{c}),S_{1,0},
\mathcal{O}(\vec{x}_1+\vec{c}))
\end{smallmatrix}}\ar[uuuu]
&&\\
&&&&&&&&\\
&&&&&&&&\\
&&&&&&&&&&&&&&&\\
&&&&&&&&&&\ar[uuuu]&&\\
&&&&&&&&&&&\\
&&&&&&\vdots&&&&\vdots&&&&\vdots&&\\
  }\]
\end{center}

There is a decomposition of  $\Sigma(\mathcal{D})$ and hence of its mutation graph:
$$\Sigma(\mathcal{D})=
\mathop{\dot{\bigcup}}\limits_{\vec{x}\in\mathbb{L}(2)}^{}
Q^{-1}_{\langle \mathcal{O}(\vec{x})\rangle}(\Sigma(\mathcal{D}/\langle \mathcal{O}(\vec{x})\rangle))\;\;\dot{\bigcup}\;\;
Q^{-1}_{\langle S_{1,0}\rangle}(\Sigma(\mathcal{D}/\langle S_{1,0}\rangle))\;\;\dot{\bigcup}\;\; Q^{-1}_{\langle S_{1,1}\rangle}(\Sigma(\mathcal{D}/\langle S_{1,1}\rangle)),
$$
where each quotient $\mathcal{D}/\langle \mathcal{O}(\vec{x})\rangle$ is of type $A_2$, and each $\mathcal{D}/\langle S_{1,j}\rangle$ is of type $\mathbb{P}^{1}$.

Furthermore, by Theorem \ref{connectedness reduction}, the connectivity of the mutation graph of $\Sigma(\mathcal{D})$ is equivalent to that of the blue-arrow component graph in Remark \ref{mutation graph of subgraphs}.
Consequently, both connectivity properties hold by Theorem \ref{finest sod and full exceptional sequence}, on account of the transitive braid group action on full exceptional sequences in $\coh \mathbb{X}(2)$ \cite{kume}.

\bigskip

{\bf Acknowledgments.}
This work was supported by the National Natural Science Foundation of China
(Grant No. 12031007) and the Natural Science Foundation of Xiamen, China (Grant No. 3502Z20227184).

\end{document}